\documentclass[reqno,11pt]{amsart}
\usepackage[colorlinks=true, linkcolor=blue, citecolor=blue]{hyperref}

\usepackage{amssymb}
\usepackage{amsmath, graphicx, rotating}
\usepackage{color}
\usepackage{soul}
\usepackage[dvipsnames]{xcolor}

\usepackage{ifthen}
\usepackage{xkeyval}
\usepackage{todonotes}
\usepackage[T1]{fontenc}
\usepackage{lmodern}
\usepackage[english]{babel}

\usepackage{ upgreek }
\usepackage{stmaryrd}
\SetSymbolFont{stmry}{bold}{U}{stmry}{m}{n}
\usepackage{amsthm}
\usepackage{float}

\usepackage{ bbm }
\usepackage{ stmaryrd }
\usepackage{ mathrsfs }
\usepackage{ frcursive }
\usepackage{ comment }

\usepackage{pgf, tikz}
\usetikzlibrary{shapes}
\usepackage{varioref}
\usepackage{enumitem}

\definecolor{rouge}{rgb}{0.7,0.00,0.00}
\definecolor{vert}{rgb}{0.00,0.5,0.00}
\definecolor{bleu}{rgb}{0.00,0.00,0.8}
\usepackage[margin=1in]{geometry}
\newtheorem{theorem}{Theorem}[section]
\newtheorem*{theorem*}{Theorem}
\newtheorem{lemma}[theorem]{Lemma}
\newtheorem{definition}[theorem]{Definition}

\newtheorem{proposition}[theorem]{Proposition}

\labelformat{hypothesis}{\textbf{M\kern-0.1mm#1}}

\newtheorem{condition}{Condition}

\newtheorem{conditionA}{A\kern-0.1mm}
\labelformat{conditionA}{\textbf{A\kern-0.1mm#1}}

\labelformat{conditionB}{\textbf{B\kern-0.1mm#1}}

\theoremstyle{definition}

\newtheorem{remark}[theorem]{Remark}

\def \eref#1{\hbox{(\ref{#1})}}

\numberwithin{equation}{section}

\def\bf#1{\mathbf{#1}}

\def\RR{\mathbb{R}}
\def\PP{\mathbb{P}}
\def\EE{\mathbb{E}}

\def\vare{{\varepsilon}}
\def \eref#1{\hbox{(\ref{#1})}}

\def\EE{\mathbb{ E}}

\def\ep{{\varepsilon}}

\begin{document}

\title[Time-inhomogeneous multi-scale SDEs with partially dissipative coefficients]
{Averaging principles for time-inhomogeneous multi-scale SDEs with partially dissipative coefficients}

\author{Xiaobin Sun\quad }
\curraddr[Sun, X.]{School of Mathematics and Statistics/RIMS, Jiangsu Normal University, Xuzhou, 221116, P.R. China}
\email{xbsun@jsnu.edu.cn}
	
\author{\quad Jian Wang\quad }
\curraddr[Wang, J.]{School of Mathematics and Statistics \& Key Laboratory of Analytical Mathematics and Applications (Ministry of Education) \& Fujian Provincial Key Laboratory of Statistics and Artificial Intelligence, Fujian Normal University, 350117, Fuzhou, P.R. China}
\email{jianwang@fjnu.edu.cn}
	
\author{\quad Yingchao Xie}
\curraddr[Xie, Y.]{School of Mathematics and Statistics/RIMS, Jiangsu Normal University, Xuzhou, 221116, P.R. China}
\email{ycxie@jsnu.edu.cn}

\begin{abstract}
In this paper, we study averaging principles for a class of time-inhomogeneous stochastic differential equations (SDEs) with slow and fast time-scales, where the drift term in the fast component is time-dependent and only partially dissipative. Under asymptotic assumptions on the coefficients, we prove that the slow component $(X^{\varepsilon}_t)_{t\ge0}$ converges strongly to the unique solution $(\bar{X}_t)_{t\ge0}$ to an averaged SDE, when the diffusion coefficient in the slow component is independent of the fast component; on the other hand, we establish the weak convergence of $(X_t^{\varepsilon})_{t\ge0}$ in the space $C([0,T];\RR^n)$ and identify the limiting process by the martingale problem approach, when the diffusion coefficient of the slow component depends on the fast component. The proofs of strong and weak averaging principles are partly based on the study of the existence and uniqueness of an evolution system of measures for time-inhomogeneous SDEs with partially dissipative drift.
\end{abstract}

\date{\today}

\maketitle
\tableofcontents
\allowdisplaybreaks
\section{Background and main results}\label{Section1}

\subsection{Background}
Multi-scale systems encompass diverse applications, where the associated processes evolve across varying temporal and spatial scales. They are prevalent in various fields such as climate modeling, material science, biology, and nonlinear oscillations, see e.g. \cite{EE}. Multi-scale systems frequently exhibit unconventional characteristics that defy common sense, to which using traditional, singular theoretical frameworks inadequate. Consequently, the exploration of multi-scale modeling systems becomes extremely important and indispensable.
The mathematical model of multi-scale stochastic differential equations (SDEs) usually can be described as follows:
$$\left\{\begin{array}{l}
\displaystyle
dX^{\varepsilon}_t=b(X^{\varepsilon}_t, Y^{\varepsilon}_t)\,dt+\sigma(X^{\varepsilon}_t, Y^{\varepsilon}_t)\,dB^1_t,\quad X^{\varepsilon}_0=x\in\mathbb{R}^n, \\
dY^{\varepsilon}_t={\varepsilon}^{-1}f(X^{\varepsilon}_t, Y^{\varepsilon}_t)\,dt+\varepsilon^{-1/2}g(X^{\varepsilon}_t, Y^{\varepsilon}_t)\,dB^2_t,\quad Y^{\varepsilon}_0=y\in\mathbb{R}^m,\\
\end{array}\right.
$$
where $(B^1_t)_{t\ge  0}$ and $(B^2_t)_{t\ge0}$ are two independent $d_1$-dimensional and $d_2$-dimensional standard Brownian motions respectively, $\varepsilon$ is a small and positive parameter describing the ratio of the time scales between the slow component $(X ^{\varepsilon}_t)_{t\ge0}$ and the fast component $(Y^{\varepsilon}_t)_{t\ge0}$, and the coefficients
$b\!:\RR^{n}\times\RR^{m} \to \RR^{n}$, $\sigma\!: \RR^{n} \times\RR^m \to\RR^{n\times d_1}$, $f\!:\RR^{n}\times\RR^{m}\to \RR^{m}$ and $g\!:\RR^{n}\times\RR^{m}\to \RR^{m\times d_2}$
are Borel measurable. The techniques employed in this context are commonly known as averaging and homogenization methods, see e.g. \cite{FW2012, K1968, PS2008} and the references therein. For example, as discussed in Khasminskii's pioneer work \cite{K1968}, it is assumed that there exist $ \bar{b}\!: \mathbb{R}^n\rightarrow \mathbb{R}^n$ and $\bar A\!: \mathbb{R}^n\rightarrow \mathbb{R}^{n\times n}$ such that for all $T\ge0$, $x\in \RR^n$ and $y\in \RR^m$,
\begin{align}
\sup_{t\ge0}\left|\frac{1}{T}\int^{t+T}_t \EE b(x, Y^{x,y}_s)\,ds -\bar b(x)\right|\le  \alpha(T)(1+|x|^2)\label{e:C1}
\end{align}
and
\begin{align}
\sup_{t\ge0}\left|\frac{1}{T}\int^{t+T}_t \EE (\sigma\sigma^{\ast})(x, Y^{x,y}_s)\,ds -\bar A(x)\right|\le  \alpha(T)(1+|x|^2),\label{e:C2}
\end{align}
where $\alpha: \RR_{+}:=[0,\infty)\to \RR_{+}$ satisfies that $\alpha(T)\to 0$ as $T\rightarrow \infty$, and $(Y^{x,y}_t)_{t\ge  0}$ is the unique solution to the following frozen SDE
\begin{equation}\label{e:FSDE}
dY^{x,y}_{t}=f(x, Y^{x,y}_{t})\,dt+g(x,Y^{x,y}_{t})\,dB^2_t,\quad Y^{x,y}_{0}=y.
\end{equation}
Then, the (weak) averaging principle claims that, for any $T>0$, the slow component $(X^{\varepsilon}_t)_{t\in [0,T]}$ converges weakly to the solution $(\bar X_t)_{t\in [0,T]}$ of the following averaged SDE
\begin{equation}\label{e:ASDE}
d\bar{X}_{t}=\bar{b}(\bar{X}_{t})\,dt+\bar{\sigma}(\bar{X}_t)\,d\bar B_t, \quad X_{0}=x
\end{equation}
as $\varepsilon\to 0$, where $\bar{\sigma}(x):=(\bar A(x))^{1/2}$, i.e., $\bar{\sigma}^2(x)=\bar A(x)$ for all $x\in \RR^n$, and $(\bar B_t)_{t\ge0}$ is a standard $n$-dimensional Brownian motion.
Undoubtedly, a key step for understanding the averaging principle lies in comprehending the corresponding averaged SDE \eqref{e:ASDE}, and the most important thing is how to accurately describe the averaged coefficients involved in. In particular, the averaging coefficients depend on the invariant measure of the frozen SDE \eqref{e:FSDE}. For example, if the frozen SDE \eqref{e:FSDE} admits a unique invariant measure $\mu^{x}$, then
$$
\bar{b}(x)=\int_{\mathbb{R}^m}b(x,y)\,\mu^{x}(dy),\quad
\bar A(x)=\int_{\mathbb{R}^m}(\sigma\sigma^{*})(x,y)\,\mu^{x}(dy).
$$
Currently most research results focus on the case where the frozen SDEs are time-homogeneous, because of the rich theoretical research on the existence and uniqueness of invariant measures in this case, see e.g. \cite{B2022,CWW2024,CHR2024,GKK2006,HL2020,HLLS2023,L2010,LRSX2020,PIX2021,RSX2021,SXW2022,SSWX2024,SXX2022,V1991, XLM2018}.

The solution to a time-nonhomogeneous frozen SDE exhibits richer dynamic characteristics, leading to the understanding of the so-called equilibrium point extremely complex. More precisely, consider the following time-dependent SDE:
$$
dY^{x,y}_{t}=f(t, x, Y^{x,y}_{t})\,dt+g(t,x,Y^{x,y}_{t})\,dB^2_t,\label{TDFrozenE}
$$
which corresponds to the following multi-scale nonautonomous SDE:
\begin{equation}\left\{\begin{array}{l}\label{Equation0}
\displaystyle
d X^{\varepsilon}_t = b(X^{\varepsilon}_{t}, Y^{\varepsilon}_{t})\,dt+\sigma(X^{\varepsilon}_{t},Y^{\varepsilon}_{t})\,d B^1_t,\quad X^{\varepsilon}_0=x\in\RR^{n}, \vspace{2mm}\\
\displaystyle d Y^{\varepsilon}_t =\varepsilon^{-1}f(t/\varepsilon,X^{\varepsilon}_t,Y^{\varepsilon}_{t})\,dt+\varepsilon^{-1/2}g(t/\varepsilon,X^{\varepsilon}_t,Y^{\varepsilon}_{t})\,d B^2_t,\quad Y^{\varepsilon}_0=y\in\RR^{m}.
\end{array}\right.
\end{equation}
As far as we know, there are very limited existing works on the time-inhomogeneous stochastic system such like \eref{Equation0}. If the coefficients $f$ and $g$ are time periodic, Wainrib \cite{W2013} and Uda \cite{U2021} studied the strong averaging principle for the stochastic system \eref{Equation0} when $\sigma\equiv 0$. Cerrai and Lunardi \cite{CL2017} explored the weak averaging principle for time-inhomogeneous slow-fast stochastic reaction diffusion equations, when the coefficients of the fast component satisfy the almost periodic condition.
For uniformly dissipative coefficients $f$ and $g$,  recently we used the technique of nonautonomous Poisson equations to study the strong and weak averaging principle for \eref{Equation0} in \cite{SWX2024}.

\subsection{Main results}\label{sec.prelim}
In this paper, we are interested in the following nonautonomous stochastic system
\begin{equation}\left\{\begin{array}{l}\label{Equation}
\displaystyle
d X^{\varepsilon}_t = b(t/\varepsilon,X^{\varepsilon}_{t}, Y^{\varepsilon}_{t})\,dt+\sigma(t/\varepsilon,X^{\varepsilon}_{t},Y^{\varepsilon}_{t})\,d W^1_t,\quad X^{\varepsilon}_0=x\in\RR^{n}, \vspace{2mm}\\
\displaystyle d Y^{\varepsilon}_t =\varepsilon^{-1}f(t/\varepsilon,X^{\varepsilon}_t,Y^{\varepsilon}_{t})\,dt+\varepsilon^{-1/2}\,d W^2_t,\quad Y^{\varepsilon}_0=y\in\RR^{m}.
\end{array}\right.
\end{equation}
Here, $(W^1_t)_{t\ge  0}$ and $(W^2_t)_{t\ge0}$ are mutually independent $d$-dimensional and $m$-dimensional standard Brownian motions respectively, on a complete probability space $(\Omega, \mathscr{F}, \mathbb{P})$ with the natural filtration $\{\mathscr{F}_{t},t\ge 0\}$ generated by $(W^1_t)_{t\ge  0}$ and $(W^2_t)_{t\ge0}$.  The coefficients $b: \RR_{+}\times\RR^n\times\RR^m\rightarrow \RR^{n}$, $\sigma: \RR_{+}\times\RR^n\times \mathbb{R}^m\rightarrow \RR^{n\times d}$ and $f: \RR_{+}\times\RR^n\times\RR^m\rightarrow \RR^{m}$ are measurable. Meanwhile, the coefficient $f(t,x,y)$ enjoys the partial dissipative property with respect to $y$; see \ref{A21} below. Such property seems to be novel within the context of multi-scale SDEs and even in the time-independent framework.

We next provide the assumptions on the coefficients of the system  \eqref{Equation}, and show the main results of the paper, including the strong and weak averaging principles for \eqref{Equation}. First of all, the following three conditions are always assumed throughout the paper.
\begin{conditionA}\label{A1} \it
There exists a constant $C_0>0$ such that for any $t\ge 0$, $x_1,x_2\in \RR^n$ and $y_1,y_2\in \RR^m$,
\begin{align}
&|b(t,x_{1},y_1)-b(t,x_{2},y_2)|+\|\sigma(t,x_1,y_1)-\sigma(t,x_2,y_2)\|\le  C_0(|x_{1}-x_{2}|+|y_1-y_2|),\label{ConA11}\vspace{3mm}\\
&|b(t,x_1,y_1)|\le  C_0(1+|x_1|+|y_1|),\quad \|\sigma(t,x_1,y_1)\|\le  C_0(1+|x_1|).  \label{ConA12}
\end{align}
\end{conditionA}

\begin{conditionA} \label{A2}
There exist constants $C,K$ and $r_0>0$ such that for all $t\ge0$, $x_1,x_2\in\RR^n$ and $y_1,y_2\in \RR^m$,
\begin{equation}\label{A21}\begin{split}
&\langle f(t,x_1,y_1)-f(t,x_2,y_2),y_1-y_2\rangle \\
&\le\begin{cases} C|y_1-y_2|^2+C|x_1-x_2||y_1-y_2|,\quad & |y_1-y_2|\le r_0, \\
-K|y_1-y_2|^2+C|x_1-x_2||y_1-y_2|,\quad &|y_1-y_2|>r_0.\end{cases}
\end{split}\end{equation}
Moreover,
there exists a constant $C_*>0$ such that for all $t\ge0$ and $x\in \RR^n$,
\begin{align}
|f(t,x,{\bf0})|\le  C_*(1+|x|) .\label{A22}
\end{align}
\end{conditionA}

\begin{conditionA}\label{A3} \it
There exists $\hat{b}: \mathbb{R}^n\times \mathbb{R}^m\rightarrow  \mathbb{R}^{n}$ such that for any $T\ge  0$, $x\in \mathbb{R}^n$ and $y\in\mathbb{R}^m$,
\begin{align}\label{Cb}
\sup_{t\ge  0}\frac{1}{T}\left|\int^{t+T}_t (b(s,x,y)-\hat{b}(x,y))\,ds\right|\le  \phi_1(T)(1+|x|+|y|).
\end{align}
Moreover, there exist $k\ge  0$ and $\bar{f}: \mathbb{R}^n\times \mathbb{R}^m\rightarrow  \mathbb{R}^{m}$ such that for all $T\ge0$, $x\in \RR^n$ and $y\in \RR^m$,
\begin{align}
|f(T,x,y)-\bar{f}(x,y)|\le  \phi_2(T)(1+|x|+|y|^k).\label{Cf}
\end{align}
Here, for all $i=1,2$, $\phi_i\ge  0$ defined on $\RR_{+}$ satisfies $\phi_i(T)\rightarrow 0$ as $T\to \infty$.
\end{conditionA}

It is well known that under Assumptions \ref{A1} and \ref{A2}, the SDE \eref{Equation} admits a unique strong solution $(X_t^\varepsilon,Y_t^\varepsilon)_{t\ge0}$, e.g.\ see  \cite[Theorem 3.1.1]{LR2015}.
On the other hand, \eqref{A21} is called a partially dissipative condition in the literature (see \cite{Eberle2014, LW2016}).
It is obvious that the condition \eqref{A21}  is weaker than the classical uniformly dissipative condition.
By applying the elementary inequality $ab\le \delta a^2+C_\delta b^2$ for all $a,b\ge0$ and $\delta>0$, we can see from \eqref{A21} and \eqref{A22} that for all $t\ge0$, $x\in \RR^n$ and $y\in \RR^m$,
\begin{equation}\label{A23}\begin{split}
\langle f(t,x,y),y\rangle&=\langle f(t,x,y)-f(t,x,{\bf0}),y\rangle+\langle f(t,x,{\bf0}),y\rangle\\
&\le -K|y|^2 +(C+K)r_0^2+|f(t,x,{\bf0})||y|\\
&\le  -K|y|^2 +(C+K)r_0^2+C_*(1+|x|)|y|\\
&\le  -K_1|y|^2+K_2(1+|x|^2),
\end{split}\end{equation}
where $0<K_1<K$ and $K_2>0$.
Two examples fulfilling Assumption \ref{A2} are as follows:
$f(t,x,y)=\alpha(t)\nabla V(y)+f_0(x)$ with $V(y)=-|y|^{2\gamma}$ and $\gamma>1$ (see \cite[Example 1.7]{LW2016}), and $f(t,x,y)=\alpha(t) (V_1(y)+V_2(y))+f_0(x)$ with $V_1:\RR^m\to \RR^m$ satisfying the uniformly dissipative condition and $V_2:\RR^m\to \RR^m$ being a bounded and Lipschitz continuous function, where $\inf_{t\ge0} \alpha(t)>0$ and $f_0:\RR^n\to \RR^m$ is a globally Lipschitz continuous function. Furthermore, the fast component $(Y_t^\varepsilon)_{t\ge0}$ is related to the following frozen SDE through a scaling transformation
\begin{equation*}\label{e:FR}
d Y_{t}=f(t,x, Y_{t})\,dt+d W^2_t.
\end{equation*}
This together with \eqref{Cf} indicate that the frozen SDE above and the fast component $(Y_t^\varepsilon)_{t\ge0}$ are closely connected with the following
SDE:
\begin{equation}\label{e:IN}
d\bar Y_{t}=\bar{f}(x,\bar Y_{t})\,dt+d W^2_t.
\end{equation}

In order to study the asymptotic behavior of $(X^{\varepsilon}_t)_{t\ge0}$, we further need the assumption on the diffusion coefficient $\sigma(t,x,y)$, which is much more complex than that for the drift coefficient $b(t,x,y)$ as taken in Assumption {\rm\ref{A3}}.
Most importantly,
we should emphasize that in order to obtain
the asymptotic behavior of $(X^{\varepsilon}_t)_{t\ge0}$ in the strong sense as $\varepsilon\rightarrow 0$, we shall require that $\sigma(t,x,y)=\sigma(t,x)$; that is, the diffusion coefficient $\sigma(t,x,y)$ is independent of $y$. This crucial assumption ensures the validity of the strong averaging principle; see \cite[Section 4.1]{L2010}. More explicitly, we need
the following condition.
\begin{conditionA}\label{A4}
Suppose that $\sigma(t,x,y)=\sigma(t,x)$, and that there exists $\bar{\sigma}: \mathbb{R}^n\rightarrow  \mathbb{R}^{n\times d}$ such that for any $T\ge  0$ and $x\in \mathbb{R}^n$,
\begin{align}\label{Csigam}
\sup_{t\ge  0}\frac{1}{T}\int^{t+T}_t \|\sigma(s,x)-\bar{\sigma}(x)\|^2\,ds\le  \phi_3(T)(1+|x|^2),
\end{align}
where $\phi_3\ge  0$ defined on $\RR_{+}$ satisfies $\phi_3(T)\rightarrow 0$ as $T\to \infty$.
\end{conditionA}

Below is our first main result.

\begin{theorem}{\rm(\textbf {Strong averaging principle})}\label{main result 2}
Suppose that Assumptions {\rm\ref{A1}}, {\rm\ref{A2}}, {\rm\ref{A3}} and {\rm\ref{A4}} hold. Then, for any initial value $(x,y)\in\RR^{n}\times\RR^{m}$ and $T>0$,
\begin{align}
\lim_{\varepsilon\rightarrow 0}\sup_{t\in[0,T]}\mathbb{E}|X_{t}^{\varepsilon}-\bar{X}_{t}|^2=0,\label{R2}
\end{align}
where $(\bar{X}_t)_{t\ge0}$ is the unique strong solution to the following averaged SDE \begin{equation}\label{e:ASDE1}
d\bar{X}_{t}=\bar{b}(\bar{X}_t)\,dt+\bar\sigma(\bar{X}_t)\,d W^1_t,\quad\bar{X}_{0}=x.
\end{equation}
Here, $\bar{b}(x)=\displaystyle\int_{\RR^{m}}\hat b(x,y)\,\mu^x(dy)$, and $\mu^x$ is the unique invariant measure of the  SDE \eqref{e:IN}.
\end{theorem}

For the general diffusion coefficient $\sigma(t,x,y)$, we will turn to investigate the weak averaging principle. For this purpose, we take the following assumption.
\begin{conditionA}\label{A5}
There exists a function $\Sigma: \mathbb{R}^n\times\mathbb{R}^m\rightarrow  \mathbb{R}^{n\times n}$ such that for any $T\ge  0$, $x\in \mathbb{R}^n$ and $y\in \RR^m$,
\begin{align}\label{Csigam1}
\sup_{t\ge0}\frac{1}{T}\left\|\int^{t+T}_t\left[\left(\sigma\sigma^{\ast}\right)(s,x,y)-\Sigma(x,y)\right]ds\right\|\le  \phi_4(T)(1+|x|^2),
\end{align}
where $\phi_4\ge  0$ defined on $\RR_{+}$ satisfies $\phi_4(T)\rightarrow 0$ as $T\to \infty$. Moreover,
\begin{align}
\inf_{t\ge  0,x\in\RR^n,y\in\RR^m,z\in \RR^n\backslash\{\textbf{0}\}}\frac{\langle (\sigma\sigma^{\ast})(t,x,y)\cdot z, z\rangle}{|z|^2}>0. \label{NonD}
\end{align}
\end{conditionA}

\begin{theorem}{\rm(\textbf {Weak averaging principle})}\label{main result 3}
Suppose that Assumptions {\rm\ref{A1}}, {\rm\ref{A2}}, {\rm\ref{A3}} and {\rm\ref{A5}} hold. Then, for any initial value $(x,y)\in\RR^{n}\times\RR^{m}$ and $T>0$, $\{X^{\varepsilon}\}_{\varepsilon\in(0,1]}:=\{(X^{\varepsilon}_t)_{t\ge0}\}_{\varepsilon\in(0,1]}$ converges weakly to $(X_t)_{t\ge0}$ in $C([0,T];\RR^n)$ as $\varepsilon\to 0$, where $(X_t)_{t\ge0}$ is the unique strong solution to the following averaged SDE
\begin{align}\label{e:ASDE2}
d X_{t}=\bar{b}(X_{t})dt+\Theta(X_t)d\bar{W}_t,\quad X_{0}=x\in \mathbb{R}^n,
\end{align}
where $\bar{b}$ is defined in  the SDE \eqref{e:ASDE1}, $\Theta(x):=\bar{\Sigma}^{1/2}(x)$ with $\bar{\Sigma}(x):=\displaystyle\int_{\RR^m}\Sigma(x,y)\,\mu^{x}(dy)$, and  $(\bar{W}_t)_{t\ge  0}$ is a standard $n$-dimensional Brownian motion.
\end{theorem}

At the end of this part, we  give some specific examples on the coefficients $b$ and $f$ that satisfy Assumption {\rm\ref{A3}}. Similarly, one can take examples for the diffusion coefficient $\sigma(t,x,y)$ that satisfies Assumptions \ref{A4} and \ref{A5} respectively.
\begin{itemize}
\item[{\rm(i)}] Suppose that $|b(t,x,y)-\hat b(x,y)|\le  \psi_1(t)(1+|x|+|y|)$ for all $t\ge0$, $x\in \RR^n$ and $y\in \RR^m$, where $\psi_1:\RR_{+}\to \RR_{+}$ is a locally integrable function and satisfies $\lim_{t\to \infty}\psi_1(t)=0$. Then \eref{Cb} holds with $\phi_1(T)=\sup_{t\ge  0}\frac1 T \int^{t+T}_t \psi_1(s)\,ds$, which converges to $0$ as $T \to \infty$; see Lemma \ref{Pro4.4} below.

\noindent If $b(\cdot,x,y)$ is $\tau$-periodic for any $x\in\RR^n$ and $y\in\RR^m$; that is, for all $t\ge0$, $x\in \RR^n$ and $y\in \RR^m$, $b(t+\tau, x,y)=b(t,x,y)$, then \eref{Cb} holds with $\hat{b}(x,y):=\frac{1}{\tau}\int^\tau_0 b(t,x,y)\,dt$ and  $\phi_1(T)=\frac{C}{T}$; see \cite[Lemma 5.1]{SWX2024}.

\item[{\rm(ii)}] If $f(T,x,y)=\alpha(T)+\bar f(x,y)$  with $\alpha:\RR_{+}\to \RR$ satisfying $\lim_{T\to \infty}|\alpha(T)|=0$, then \eref{Cf} holds with $\phi_2(T)=|\alpha(T)|$ and $k=0$.

   \noindent If $f(T,x,y)=\beta(T)\bar f(x,y)$ with $|\bar f(x,y)|\le  C(1+|x|+|y|^k)$ for some $k\ge  0$ and $\beta:\RR_{+}\to (0,\infty)$ satisfying  $\lim_{T\to \infty}\beta(T)=1$,  then \eqref{Cf} holds with $\phi_2(T)=|\beta(T)-1|$.
\end{itemize}

\subsection{Approaches}

As we discussed before, the key idea to study the averaging principle for \eqref{Equation} is to provide precise characterizations of the averaged coefficients. The frozen SDE
corresponding to \eqref{Equation} is
given by  \begin{align}\label{e:adad}
d Y_t =f(t,x,Y_{t})dt+d W^2_t,
\end{align}
which associates with a time-nonhomogeneous transition semigroup $(P^x_{s,t})_{t\ge  s}$. General speaking, there is no invariant measure any more. A natural generalization of the notation for invariant measure corresponding to $(P^x_{s,t})_{t\ge  s}$ is an evolution system of measures (see e.g. \cite{DR2006,DR2008,LLZ2010}). Recall that, $\{\mu^x_t\}_{t\in \RR}$ is an evolution system of measures  for $(P^x_{s,t})_{t\ge  s}$, if
$$
\int_{\RR^m}P^x_{s,t} \varphi(y)\,\mu^x_s(dy)=\int_{\RR^m}\varphi(y)\,\mu^x_t(dy),\quad -\infty<s\le  t<+\infty,\varphi\in C_b(\RR^m).
$$
Hence, in order to achieve this, we need to extend the SDE \eqref{e:adad} to define for all $t\in \RR$ instead of $t\in \RR_+$, which particularly requires the coefficient $f(t,x,y)$ and the noise $W^2_t$ to be well-defined for all $t\in \RR$; see Subsection \ref{Sub2.1} below for the details.

It is worthwhile to highlight that, when the coefficient $f(t,x,y)$ is partially dissipative with respect to $y$, we will utilize the so-called asymptotic reflection coupling for different drifts in order to establish the weak and strong ergodicity for the frozen SDE \eqref{e:adad} under the $L_1$-Wasserstein distance. To the best of our knowledge, this aspect in the study of the averaging principle appears to be new, and has not been explored before. We emphasize that, due to the partially dissipative drift, it seems that the method via the Poisson equation  employed in \cite{SWX2024} is not applicable in the current context, and in particular it seems challenging to derive the second derivative of $P^x_{s,t} \varphi(y)$ with respect to $y$, which is essential for constructing the solution to the Poisson equation.

Below we briefly state the approaches to the proofs of Theorems \ref{main result 2} and \ref{main result 3}, which partly indicate the novelties of our paper.

\subsubsection{\textbf{Strong averaging principle}}  Our first contribution lies in the study of the asymptotic behavior for $\{X^\varepsilon\}_{\varepsilon\in(0,1]}$ in the strong sense as $\varepsilon\to 0$, in the case of $\sigma(t,x,y)=\sigma(t,x)$ for all $t\ge0$, $x\in \RR^n$ and $y\in \RR^m$. To accomplish this objective, we will proceed in two steps.

\textbf{Step 1:} Based on the observation above, it is very natural to introduce the averaged drift coefficient as follows:
$$
\bar{b}(t,x)=\int b(t,x,y)\,\mu^x_t(dy),
$$
which leads to the following averaged SDE:
\begin{align}
d\bar{X}^{\varepsilon}_{t}=\bar{b}(t/\varepsilon,\bar{X}^{\varepsilon}_t)\,dt+\sigma(t/\varepsilon,\bar{X}^{\varepsilon}_t)\,d W^1_t,\quad\bar{X}^{\varepsilon}_{0}=x.\label{MAE}
\end{align} Here, $\{\mu^x_t\}_{t\in \RR}$ is an evolution system of measures  for the Markov semigroup $(P^x_{s,t})_{t\ge  s}$ corresponding to the extension of the frozen SDE \eqref{e:adad}.
Thus, we will first demonstrate that
\begin{align}
\sup_{t\in [0,T]}\EE|X^{\varepsilon}_t-\bar{X}^{\varepsilon}_{t}|^2\rightarrow 0, \quad \varepsilon\rightarrow 0,\label{Step1}
\end{align}
where $(\bar{X}^{\varepsilon}_t)_{t\ge0}$ is the solution to the SDE \eqref{MAE}.

\textbf{Step 2:} It is worth mentioning that the coefficients of
the SDE \eqref{MAE} still depend on $\varepsilon$.
In order to achieve coefficients in the final averaged SDE that are independent of $\varepsilon$, it is sufficient to guarantee that
\begin{align}
\frac{1}{T}\displaystyle\int^{t+T}_t \bar b(s,x)ds\rightarrow \bar{b}(x),\quad T \to \infty \label{e:Cb}
\end{align}
uniformly for $t$, and the condition \eqref{Csigam} holds.
As stated in \eqref{Step1}, to deal with the averaging principle for the nonautonomous stochastic system
\eqref{Equation}, we shall turn to a study of the averaged SDE \eqref{MAE}, as $\vare\to 0$.
Starting from \eqref{MAE}, the condition \eqref{e:Cb} is essentially equivalent to \eqref{e:C1} by taking $Y^{x,y}_t=t$, see \cite[(6)]{V1991}. For this purpose, in contrast to the time-periodic case imposed on the coefficients (see \cite{U2021, W2013}), we introduce here a reasonable Assumption \ref{A3}.
Such assumption guarantees that $\mathbb{W}_1(\mu^x_t, \mu^x)\rightarrow 0$ as $t\to \infty$, and it in turn indicates that \eqref{e:Cb} holds with $\bar{b}(x)=\int_{\RR^{m}}\hat b(x,y)\,\mu^x(dy)$, thanks to \eqref{Cb}. Here, $\mu^x$ is the unique invariant measure of the SDE \eqref{e:IN}.
Thus, in this step we will prove that
\begin{align}
\sup_{t\in [0, T]}\EE|\bar X_{t}^{\varepsilon}-\bar{X}_{t}|^2\rightarrow 0,\quad  \varepsilon\rightarrow 0, \label{Step2}
\end{align}
where $(\bar{X}_t)_{t\ge0}$ is the solution to the ultimate averaged SDE \eqref{e:ASDE1} given in Theorem \ref{main result 2}. Combining \eqref{Step1} with \eqref{Step2}, we achieve the strong averaging principle.

\subsubsection{\textbf{Weak averaging principle}}  It is crucial to note that the assumption  $\sigma(t,x,y)=\sigma(t,x)$
is essential to establish the strong convergence. In the absence of this assumption, the strong convergence may not hold. However, we can expect the weak convergence in the general setting. For this,  Assumption \ref{A5} on $\sigma(t,x,y)$ is needed, which is related to \eqref{e:C2}; see \cite[(7)]{V1991}. Hence, our second significant contribution lies in proving that $\{X^\varepsilon\}_{\varepsilon\in(0,1]}$ converges weakly to $X$ in $C([0,T];\RR^n)$ as $\varepsilon\to 0$, where $X$ is the unique solution to the
SDE \eqref{e:ASDE2} given in Theorem \ref{main result 3}.

To establish the weak convergence of $\{X^\varepsilon\}_{\varepsilon>0}$ in $C([0,T];\RR^n)$ for every $T>0$, we proceed in the following three steps.

\textbf{Step 1:}  We show that the family of the processes $\{X^{\varepsilon}\}_{\varepsilon\in (0,1]}$ is tight in $C([0,T];\mathbb{R}^n)$. This ensures that for any sequence $\{\varepsilon_k\}_{k\ge1}$ tending to zero, there exists a subsequence of $\{\varepsilon_k\}_{k\ge  1}$ (which is still denoted by $\{\varepsilon_k\}_{k\ge1}$ for simplicity) such that $X^{\varepsilon_k}$ converges weakly to some accumulation point $X$ in $C([0,T];\mathbb{R}^n)$ as $k\to \infty$.

\textbf{Step 2:} To identify the weak limit $X$, we utilize the martingale problem approach. We show that the limiting process $X$ is a solution to the martingale problem associated with the SDE \eqref{e:ASDE2}.

\textbf{Step 3:} We rely on the uniqueness of the strong solution to the SDE \eqref{e:ASDE2} to conclude that the entire family $\{X^{\varepsilon}\}_{\varepsilon\in (0,1]}$ converges weakly to $X$ as $\varepsilon\to0$.

\ \

We shall mention that in our paper we are only concerned with the stochastic system \eqref{Equation}, where the fast component $(Y_t^\ep)_{t\ge0}$ is driven by additive Brownian motion. This is only for the simplicity of the notations. Indeed, the approach of our paper works for the case that the fast component $(Y_t^\ep)_{t\ge0}$ has the multiplicative noise. The main technical differences are due to proofs of Lemma \ref{L3.1}(2) and Lemma \ref{Lemma6.2} below. In order to obtain the corresponding assertions, one can make use of the mixture of the asymptotic coupling by reflection and the synchronous coupling; see \cite{PW2006, Wang2} for the details.

\ \

The remainder of the paper is structured as follows. In Section 2, we will study the existence and uniqueness of an evolution system of measures for the time-dependent SDE \eqref{e:adad} with partially dissipative drift. This involves analyzing the weak ergodicity of the SDE \eqref{e:adad} under the $L_1$-Wasserstein distance. Here, we also establish a priori estimates of the stochastic system \eqref{Equation}, and a quantitative difference between the fast component and its  auxiliary process. Sections 3 and 4 are devoted to the proofs of the strong averaging principle and the weak averaging principle, respectively.
Throughout this paper, we use $C$ to represent constants whose value may vary from line to line. We use $C_{T}$ (resp. $C_{p,T}$) to emphasize that the constant depends on $T$ (resp. $p$ and $T$).

\section{Preliminaries}

Denote by $|\cdot|$ and  $\langle\cdot, \cdot\rangle$ the Euclidean vector norm and the usual Euclidean inner product, respectively.
Let $\|\cdot\|$ be the matrix or the operator norm if there is no confusion possible.  Let $\textbf{0}$ be the zero row vector in $\RR^m$ or $\RR^n$.
Let $\mathscr{P}(\RR^m)$ be the set of all probability measures on $(\RR^m, \mathscr{B}(\RR^m))$. The {Wasserstein distance} $\mathbb{W}_d$ associated with a metric $d(x,y)$ on $\RR^m$ is defined by
$$
\mathbb{W}_{d}(\mu_1,\mu_2):=\inf_{Q\in \mathscr{C}(\mu,\nu)} \int_{ \RR^m\times  \RR^m}d(x,y)\,Q(d x,d y),\quad \mu,\nu\in \mathscr{P}(\RR^m),
$$
where $\mathscr{C}(\mu,\nu)$ is the set of all coupling measures $Q$ on $ \RR^m\times \RR^m$ of $\mu$ and $\nu$, i.e., for any $A\in \mathscr{B}(\RR^m)$,
$$Q(A\times \RR^m)=\mu(A),\quad Q(\RR^m\times A)=\nu(A).$$
If $d(x,y)=|x-y|$, then $\mathbb{W}_{d}$ is the standard $L^1$-{Wasserstein distance} which is denoted by $\mathbb{W}_1(\mu_1,\mu_2)$ throughout the paper. Let
$$\mathscr{P}_1(\RR^m):=\left\{\mu\in\mathscr{P}(\RR^m) : \int_{\RR^m} |z|\,\mu(dz)<\infty \right\}.$$
Then, $\left(\mathscr{P}_1(\RR^m),\mathbb{W}_{1}\right)$  is  a complete metric space, see, e.g., \cite[Theorem 5.4]{Chen1992}. Furthermore, by \cite[Theorem 5.10]{Chen1992},
\begin{align*}
\mathbb{W}_1(\mu,\nu)=\sup\left\{|\mu(f)-\nu(f)|: f\in \mathrm{Lip}(\RR^m)\hbox{ with } {\rm Lip}(f)\le  1\right\},
\end{align*}
where $\mathrm{Lip}(\RR^m)$ is the set of Lipschitz continuous functions on $\RR^m$ and ${\rm Lip}(f)$ denotes the Lipschitz constant of $f\in \mathrm{Lip}(\RR^m)$.

In this section, we first study the existence and the uniqueness of an evolution system of measures for the extension of the time-dependent SDE \eqref{e:adad} with partially dissipative drift $f(t,x,\cdot)$. Secondly, we establish uniform bounds for the moments of the solution ($X_{t}^{\varepsilon}, Y_{t}^{\varepsilon})_{t\ge0}$ to the stochastic  system (\ref{Equation}). Finally, we construct an auxiliary process $(\hat Y^{\varepsilon}_t)_{t\ge0}$ and show the difference between the processes $(Y^{\varepsilon}_t)_{t\ge0}$ and $(\hat Y^{\varepsilon}_t)_{t\ge0}$ in the weak sense, which is crucial in the proofs of our main results.

\subsection{An evolution system of measures for SDEs with partially dissipative drift}\label{Sub2.1}

In order to study the existence and the uniqueness of an evolution system of measures for the frozen SDE \eqref{e:adad}, we need to extend it to define for all $t\in \RR$ instead of $t\in \RR_+$. To do this, letting
$(\tilde{W}^2_t)_{t\ge 0}$
be an independent copy of Brownian motion $(W^2_t)_{t\ge 0}$,  we define
$$
{W}^2_t=\begin{cases}
W^2_t,\quad\,\, t\ge 0,\\
\tilde{W}^2_{-t}, \quad t< 0.
\end{cases}
$$
For the coefficient $f(t,x,y)$, we will use its symmetric extension, i.e., for all $x\in \RR^n$ and $y\in \RR^m$,
$$
f(t,x,y)=
\begin{cases}
f(t,x,y),\quad\,\,  t\ge 0,\\
f(-t,x,y),\quad  t< 0.
\end{cases}
$$
In particular, the extended function $f$ satisfies Assumption {\rm\ref{A2}} for all $t\in\RR$, $x\in \RR^n$ and $y\in \RR^m$. It is important to note that the extended function $f$
that meets Assumption {\rm\ref{A2}} is not unique.  Nevertheless, we here simply utilize the symmetric extension, as it is enough for our purpose.

Now, we consider the following time-inhomogeneous frozen SDE:
\begin{equation}
dY_{t}=f(t,x,Y_{t})dt+d W^2_t,\quad  Y_{s}=y\in \RR^{m}, x\in \RR^n,
s,t\in \RR \hbox{ with }
 t\ge s. \label{FrozenE}
\end{equation}
Then, for any fixed $x\in \RR^n$,  any initial value $y\in\RR^m$ and $s\in\mathbb{R}$, the SDE \eref{FrozenE} has a unique strong solution, denoted by $(Y^{s,x,y}_t)_{t\ge s}$, which is a time-inhomogeneous Markov process. For any bounded measurable function $\varphi$ on $\RR^m$,
the associated inhomogeneous Markov semigroup is given by
$$
P^{x}_{s,t}\varphi(y):=\EE\varphi(Y_{t}^{s,x,y}), \quad x\in\RR^n,y\in\RR^{m},  s,t\in \RR \hbox{ with } t\ge s.
$$
We have $P^{x}_{s,t}=P^{x}_{s,r}\circ P^{x}_{r,t}$ for $s\le  r\le  t$, and $\delta_y P^{x}_{s,t}$ is the distribution of $Y^{s,x,y}_t$.

\begin{lemma}\label{L3.1}
Suppose that Assumption {\rm\ref{A2}} holds. Then the following statements hold.
\begin{itemize}
\item[{\rm (1)}] For any $p\ge 4$, there exist constants $C_{1,p}$ and $C_{2,p}>0$ such that for all $t\ge s$, $x\in \RR^n$ and $y\in \RR^m$,
\begin{equation}
\EE|Y_{t}^{s,x,y} |^{p} \le e^{-C_{1,p}(t-s)}|y|^p+C_{2,p}(1+|x|^p);\label{uniformEY}
\end{equation}
\item[{\rm (2)}] There exist constants $\beta$ and $C>0$ such that for any $t \ge s$, $x_1,x_2\in\RR^{n}$ and $y_1,y_2\in\RR^{m}$,
\begin{align}
\mathbb{W}_1 (\delta_{y_1}P^{x_1}_{s,t}, \delta_{y_2}P^{x_2}_{s,t} )\le Ce^{-\beta(t-s) }|y_1-y_2|+C|x_1-x_2|.\label{FR1}
\end{align}
\end{itemize}
\end{lemma}

\begin{proof}
(1)  It follows from \eqref{FrozenE} that for all $t\ge s$, $x\in \RR^n$ and $y\in \RR^m$,
\begin{align*}
Y_{t}^{s,x,y} =y+\int_{s}^{t} f(r,x, Y_{r }^{s,x,y})dr + \int^t_s d W^2_{r}.
\end{align*}
By applying the It\^{o} formula and taking the expectation, we get that for any $p\ge 4$,
\begin{align*}
\mathbb{E}\left|Y_{t}^{s,x,y}\right|^{p}=&|y|^p\!+\!p\mathbb{E}\int_s^t\!\left|Y_{r}^{s,x,y}\right|^{p-2}\left\langle Y_{r}^{s,x,y},f(r,x,Y_{r}^{s,x,y})\right\rangle dr\!+\!\frac{p(p-1)}{2}\int_{s} ^{t}\!\EE|Y^{s,x,y}_r|^{p-2}\,dr.
\end{align*}
By \eref{A23} and Young's inequality, for some $0<K_3<K_1$, it holds that
\begin{align*}
\frac{d}{dt}\mathbb{E}|Y_{t}^{s,x,y}|^{p}=& p\mathbb{E}\left[|Y_{r}^{s,x,y}|^{p-2}\langle Y_{t}^{s,x,y},f(t,x,Y_{t}^{s,x,y})\rangle\right]+\frac{p(p-1)}{2}\mathbb{E}|Y_{t}^{s,x,y}|^{p-2} \nonumber\\
\le & -pK_1\mathbb{E}|Y_{t}^{s,x,y} |^{p}+C_p\mathbb{E}|Y_{t}^{s,x,y}|^{p-2}(1+|x|^2) \nonumber\\
\le & -pK_3\mathbb{E}|Y_{t}^{s,x,y} |^{p}+C_p(1+|x|^p).
\end{align*}
The comparison theorem yields that for all $t\ge s$, $x\in \RR^n$ and $y\in \RR^m$,
\begin{align*}
\mathbb{E}|Y_{t}^{s,x,y}|^{p}\le & e^{-p K_3(t-s)}|y|^p+C_p(1+|x|^p)\int^t_s e^{-p K_3(t-u)}du\\
\le & e^{-p K_3(t-s)}|y|^p+C_p(1+|x|^p).
\end{align*}
Hence, \eqref{uniformEY} is proved.

(2) We will make use of the so-called asymptotic reflection coupling; see \cite{DEGZ, Eberle2014} for related applications. Fix $t \ge s$, $x_1,x_2\in\RR^{n}$ and $y_1,y_2\in\RR^{m}$. Let $(W_t^{2,1})_{t\in \RR}$ and  $(W_t^{2,2})_{t\in\RR}$ be two independent $m$-dimensional Brownian motions. For any $\delta>0$, let $\pi_\delta\in C_b^2(\RR_+; [0,1])$ be such that $\pi_\delta(r)=1$ for all $r\ge \delta$ and $\pi_\delta(r)=0$ for all $r\le \delta/2$. Then, according to L\'evy's characterization of Brownian motion and the strong uniqueness of the frozen SDE \eqref{FrozenE}, we can rewrite $(Y^{s,x_1,y_1}_t)_{t\ge s}$ solving the SDE \eqref{FrozenE} with $Y^{s,x_1,y_1}_s=y_1$ as follows (with the same distribution):
$$dY^{s,x_1,y_1}_{t}\!=\!f(t,x_1,Y^{s,x_1,y_1}_{t})dt\!+\!\sqrt{\pi_\delta (|Y^{s,x_1,y_1}_{t}\!-\!Z^{s,x_2,y_2}_{t}|)}\,d W^{2,1}_t\!+\!\sqrt{1\!-\!\pi_\delta (|Y^{s,x_1,y_1}_{t}\!-\!Y^{s,x_2,y_2}_{t}|)}\,d W^{2,2}_t,$$
where $(Z^{s,x_2,y_2}_t)_{t\ge s}$ is a unique strong solution to the following SDE
 \begin{align*}
dZ^{s,x_2,y_2}_t=&f(t,x_2,Z^{s,x_2,y_2}_t)\,d t\\
&+\!  \sqrt{\pi_\delta (|Y^{s,x_1,y_1}_{t}\!-\!Z^{s,x_2,y_2}_{t}|)} \left(d W^{2,1}_t\!-\!2\frac{(Y^{s,x_1,y_1}_t\!-\!Z^{s,x_2,y_2}_t)\langle Y^{s,x_1,y_1}_t\!-\!Z^{s,x_2,y_2}_t, dW^{2,1}_t\rangle}{|Y^{s,x_1,y_1}_t-Z^{s,x_2,y_2}_t|^2}\right)\\
&+ \sqrt{1-\pi_\delta (|Y^{s,x_1,y_1}_{t}-Y^{s,x_2,y_2}_{t}|)}\,d W^{2,2}_t
\end{align*}
with $Z^{s,x_2,y_2}_s=y_2$. In particular, under Assumption {\rm\ref{A2}}, the SDE above has a unique strong solution $(Y^{s,x_1,y_1}_t,Z^{s,x_2,y_2}_t)_{t\ge s}$. L\'evy's characterization of Brownian motion also yields that  $(Y^{s,x_2,y_2}_t)_{t\ge s}$ and $(Z^{s,x_2,y_2}_t)_{t\ge s}$ have the same distribution.
Let $\{\mathscr{L}^{x_1,x_2}_t\}_{t\ge s}$ be the infinitesimal generator of $(Y^{s,x_1,y_1}_t,Z^{s,x_2,y_2}_t)_{t\ge s}$. Then, for any $h\in C^2(\RR)$, we have the following straightforward formula (see e.g., \cite[(3.15)]{PW2006})
$$
\mathscr{L}^{x_1,x_2}_t  h(|y_1\!-\!y_2|) \!=\!2 \pi_\delta (|y_1\!-\!y_2|) h''(|y_1\!-\!y_2|)\!+\!\frac{h'(|y_1\!-\!y_2|)}{ |y_1-y_2| }\langle f(t,x_1,y_1)\!-\!f(t,x_2,y_2), y_1\!-\!y_2\rangle.
$$
This along with \eqref{A21} gives us that for any $h\in C^2(\RR_+)$ with $h'\ge0$ and $h''\le 0$,
\begin{align*}
\mathscr{L}^{x_1,x_2}_t  h(|y_1-y_2|)
\le& 2 \pi_\delta (|y_1\!-\!y_2|) h''(|y_1\!-\!y_2|)\!+\! \left[(C\!+\!K)1_{\{|y_1\!-\!y_2|\le r_0\}}\!-\!K\right]\!|y_1-y_2| h'(|y_1\!-\!y_2|)\\
&+C h'(|y_1-y_2|)|x_1-x_2|.
\end{align*}

Take $d(y_1,y_2)=h(|y_1-y_2|)$ and $h(r)=1-e^{-c_1r}+c_2r$ with $c_1=2Cr_0$ and $c_2=c_1e^{-c_1r_0}.$ One has $c_2r\le h(r)\le (c_1+c_2)r$ for all $r\ge0$. Thus, for $y_1,y_2\in \RR^m$ with $|y_1-y_2|\le r_0$,
\begin{align*}
\mathscr{L}^{x_1,x_2}_td(y_1,y_2)\!\le &\pi_\delta (|y_1\!-\!y_2|)\left[-2c_1^2e^{-c_1|y_1\!-\!y_2|}\!+\!C\left(|y_1\!-\!y_2|\!+\!|x_1\!-\!x_2|\right)(c_1e^{-c_1|y_1-y_2|}\!+\!c_2)\right]\\
&+C(1-\pi_\delta (|y_1-y_2|))  \left(|y_1-y_2|+|x_1-x_2|\right)\left(c_1e^{-c_1|y_1-y_2|}+c_2\right)\\
\le&\pi_\delta (|y_1-y_2|)\left(-c_1^2e^{-c_1r_0}+4 r_0 C^2 |x_1-x_2|\right)+ 4r_0C^2|x_1-x_2|+4r_0C^2\delta\\
\le& \Big[- c_1^2e^{-c_1r_0}/\!\!\sup_{0<r\le r_0} h(r)\Big]h(|y_1\!-\!y_2|)\pi_\delta (|y_1\!-\!y_2|)\!+\!8 r_0 C^2 |x_1\!-\!x_2|\!+\!4r_0C^2\delta\\
\le &\Big[- c_1^2e^{-c_1r_0}/\!\!\sup_{0<r\le r_0} h(r)\Big]d(y_1,y_2)+8 r_0 C^2 |x_1-x_2|\\
&+4r_0C^2\delta+\Big[ c_1^2e^{-c_1r_0}/\sup_{0<r\le r_0} h(r)\Big]h(\delta),
\end{align*}
where in the first inequality we used the definitions of $c_1$ and $c_2$, and in the second and the last inequalities we used the fact that $1-\pi_\delta(r)\neq 0$ if only $r\in [0,\delta]$.
On the other hand, for $y_1,y_2\in \RR^m$ with $|y_1-y_2|>r_0$,
\begin{align*}
\mathscr{L}^{x_1,x_2}_td(y_1,y_2)\le & -c_2 K|y_1-y_2|+4 r_0 C^2 |x_1-x_2|\\
\le&  \Big[-c_2 K \inf_{r\ge r_0} (r/h(r)) \Big] h(|y_1-y_2|)+4 r_0 C^2 |x_1-x_2|\\
=&\Big[-c_2 K\inf_{r\ge r_0} (r/h(r))\Big]d(y_1,y_2)+4 r_0 C^2|x_1-x_2|.
\end{align*}

Put
$$
\beta:=\Big(c_1^2e^{-c_1r_0}/\sup_{0<r\le r_0} h(r)\Big) \wedge \Big(c_2 K\inf_{r\ge r_0} (r/h(r))\Big).$$
We have
$$\mathscr{L}^{x_1,x_2}_td(y_1,y_2)\le -\beta d(y_1,y_2)+8 r_0 C^2|x_1-x_2|+4r_0C^2\delta+\Big[ c_1^2e^{-c_1r_0}/\sup_{0<r\le r_0} h(r)\Big]h(\delta).$$
This implies that
\begin{align*}
\frac{d}{dt}\mathbb{E} d(Y_t^{s,x_1,y_1}, Z_t^{s,x_2,y_2})\le &  -\beta\mathbb{E} d(Y_t^{s,x_1,y_1}, Z_t^{s,x_2,y_2})+8 r_0 C^2|x_1-x_2|\\
&+4r_0C^2\delta+\Big[ c_1^2e^{-c_1r_0}/\sup_{0<r\le r_0} h(r)\Big]h(\delta).
\end{align*}
The comparison theorem yields that for all $t\ge s$, $x_1,x_2\in \RR^n$ and $y_1,y_2\in \RR^m$,
\begin{align*}
\mathbb{E} d(Y_t^{s,x_1,y_1}, Z_t^{s,x_2,y_2})\le & e^{-\beta(t-s) }d(y_1,y_2)+ 8 r_0 C^2|x_1-x_2|\int^t_s e^{-\beta (t-u) } du\\
&+ \Big[4r_0C^2\delta+\Big( c_1^2e^{-c_1r_0}/\sup_{0<r\le r_0} h(r)\Big)h(\delta)\Big]\int^t_s e^{-\beta (t-u) } du\\
\le & e^{-\beta(t-s)}d(y_1,y_2)+8|x_1-x_2| r_0 C^2/\beta\\
&+\Big[4r_0C^2\delta+\Big( c_1^2e^{-c_1r_0}/\sup_{0<r\le r_0} h(r)\Big)h(\delta)\Big]\Big/\beta.
\end{align*}

Finally, noting that $c_2|y_1-y_2|\le  d(y_1,y_2)\le (c_1+c_2)|y_1-y_2|$ for all $y_1,y_2\in \RR^m$, we get that for all $t\ge s$, $x_1,x_2\in \RR^n$ and $y_1,y_2\in \RR^m$,
\begin{align*}
\mathbb{W}_1\left(\delta_{y_1}P^{x_1}_{s,t}, \delta_{y_2}P^{x_2}_{s,t}\right)\le & c_2^{-1}\mathbb{W}_d\left(\delta_{y_1}P^{x_1}_{s,t}, \delta_{y_2}P^{x_2}_{s,t}\right)\\
\le  & c_2^{-1}\mathbb{E} d(Y_t^{s,x_1,y_1}, Z_t^{s,x_2,y_2})\\
\le  & c_2^{-1}(c_1+c_2)e^{-\beta(t-s)}|y_1-y_2|+8 c_2^{-1} |x_1-x_2|r_0 C^2/\beta\\
&+c_2^{-1}\Big[4r_0C^2\delta+\Big( c_1^2e^{-c_1r_0}/\sup_{0<r\le r_0} h(r)\Big)h(\delta)\Big]\Big/\beta,
\end{align*} where in the second inequality we used the fact that the distribution of $Z_t^{s,x_2,y_2}$ is $\delta_{y_2}P^{x_2}_{s,t}$.
The proof of \eqref{FR1} is complete by letting $\delta\to 0$ in the estimate above.
\end{proof}

\begin{remark}
Consider the following general frozen SDE
\begin{align}
d Y_t=f(t,U_t,Y_t)\,dt+dW^2_t,\quad Y_s=V\in\RR^m, t\ge  s,\label{GenFRZ}
\end{align}
where $(U_t)_{t\ge s}$ is an $\RR^n$-valued stochastic process, and $V$ is an $\RR^m$-valued random variable. Suppose that Assumption {\rm\ref{A2}} holds. Denote $(Y^{s,U,V}_t)_{t\ge s}$ by the unique strong solution to the SDE \eref{GenFRZ}. Following the proof of Lemma \ref{L3.1}, we can obtain that there is a constant $C>0$ such that for all $t\ge s$, $\RR^n$-valued stochastic processes $(U_1(t))_{t\ge s}$ and $(U_2(t))_{t\ge s}$, and $\RR^m$-valued random variables $V_1$ and $V_2$,
\begin{align}
\mathbb{W}_1\big(\mathscr{L}_{Y^{s,U_1,V_1}_t}, \mathscr{L}_{Y^{s,U_2,V_2}_t}\big) \le   Ce^{-\beta(t-s)}\EE|V_1-V_2|+C\sup_{s\le  u\le  t}\EE|U_1(u)-U_2(u)|.\label{REGenFRZ}
\end{align}
\end{remark}

\begin{lemma}\label{Le3.3} Suppose that Assumption {\rm\ref{A2}} holds. Then, for any $t\in \RR$ and $x\in\RR^n$, there exists $\mu^x_{t} \in \mathscr{P}_1(\RR^m)$ such that for any $s\le t$ and $y\in\RR^m$,
\begin{align}
\mathbb{W}_1(\delta_{y}P^x_{s,t}, \mu^x_t)\le  C_1\left(1+|x|+|y|\right)e^{-\beta (t-s)},\quad\sup_{t\in\RR}\int_{\RR^m}|y|\mu^x_{t}(dy)\le  C_1(1+|x|),\label{F3.2}
\end{align}
where $\beta>0$ is defined in Lemma $\ref{L3.1}$, and $C_1>0$ is independent of $t$, $s$, $x\in \RR^n$ and $y\in \RR^m$. Moreover, there is a constant $C_2>0$ such that for any $x_1,x_2\in\mathbb{R}^n$,
\begin{align}
\sup_{t\in\mathbb{R}}\mathbb{W}_1(\mu^{x_1}_t, \mu^{x_2}_t)\le  C_2|x_1-x_2|.\label{F3.3}
\end{align}
\end{lemma}

\begin{proof}
(i) Fix $t\ge s$, $x\in \RR^n$ and $y\in \RR^m$. For any $h>0$,
\begin{align*}
Y_{t}^{s-h,x,y}= Y_{s}^{s-h,x,y}+\int_{s}^{t} f(r, x, Y_{r}^{s-h,x,y} )\,dr + \int_{s}^{t}  d W^2_{r},
\end{align*}
which implies that $Y_{t}^{s-h,x,y}$ and $Y_{t}^{s,x,Y_{s}^{s-h,x,y}}$ enjoy the same distribution. By \eref{REGenFRZ} and \eref{uniformEY}, we have
\begin{align*}
\mathbb{W}_1\left(\delta_{y}P^{x}_{s,t}, \delta_{y}P^{x}_{s-h,t}\right)= & \mathbb{W}_1\big(\mathscr{L}_{Y^{s,x,y}_t}, \mathscr{L}_{Y^{s,x,Y_{s}^{s-h,x,y}}_t}\big)\nonumber\\
\le  & Ce^{-\beta(t-s) }\EE|y-Y_{s}^{s-h,x,y}|\nonumber\\
\le  & Ce^{-\beta(t-s) }(1+|x|+|y|).
\end{align*}
As a consequence, for fixed $t\in \RR$, $x\in \RR^n$ and $y\in \RR^m$, $\{\mathscr{L}_{Y^{s,x,y}_t}\}_{s\le  t}$ is a Cauchy sequence in $\mathscr{P}_1(\RR^m)$ as $s\to -\infty$, also thanks to \eqref{uniformEY}. Then, there exists  $\mu^{x,y}_t\in\mathscr{P}_1(\RR^m)$ such that for all $t\ge s$, $x\in \RR^n$ and $y\in \RR^m$,
\begin{align}
\mathbb{W}_1\left(\delta_{y}P^{x}_{s,t}, \mu^{x,y}_t\right) \le  Ce^{-\beta (t-s)}(1+|x|+|y|). \label{F3.10}
\end{align}

Next, we are going to prove that $\mu^{x,y}_{t}$ is independent of $y$. In fact, \eref{FR1} implies that for any $t\in \RR$, $x\in\RR^n$ and $y_1,y_2\in \mathbb{R} ^{m} $,
$$\lim_{s \to -\infty} \mathbb{W}_1\!\left(\delta_{y_1}\!P^{x}_{s,t}, \delta_{y_2}\!P^{x}_{s,t}\right) =0.$$
On the other hand, it holds that
$$
\mathbb{W}_1\!\left(\mu^{x,y_1}_t, \mu^{x,y_2}_t\right)\le  \mathbb{W}_1\!\left(\mu^{x,y_1}_t, \delta_{y_1}\!P^{x}_{s,t}\right)+\mathbb{W}_1\!\left(\mu^{x,y_2}_t, \delta_{y_2}\!P^{x}_{s,t}\right)+ \mathbb{W}_1\!\left(\delta_{y_1}\!P^{x}_{s,t}, \delta_{y_2}\!P^{x}_{s,t}\right).
$$
Taking $s\rightarrow -\infty$ and applying \eqref{F3.10}, we get $\mathbb{W}_1\!\left(\mu^{x,y_1}_t, \mu^{x,y_2}_t\right)=0$, and so $\mu^{x,y_1}_t=\mu^{x,y_2}_t$ in $\mathscr{P}_1(\RR^m)$. Below, we denote $\mu_t^{x,y}$ by $\mu_t^x$. Furthermore, it follows from \eqref{uniformEY} and \eqref{F3.10} that for all $x\in\RR^n$,
\begin{align*}
\sup_{t\in\RR}\int_{\RR^m}|y|\,\mu^x_t(dy)\le  C(1+|x|).
\end{align*}
Hence \eref{F3.2} holds, thanks to \eqref{F3.10} again.

(ii) By \eref{FR1} and \eref{F3.10}, we have  for any $t\in \RR$ and $x_1,x_2\in\mathbb{R}^n$,
\begin{align*}
\mathbb{W}_1(\mu^{x_1}_t, \mu^{x_2}_t)\le  &\mathbb{W}_1(\mu^{x_1}_t, \delta_{\textbf{0}}P^{x_1}_{s,t})+\mathbb{W}_1(\mu^{x_2}_t, \delta_{\textbf{0}}P^{x_2}_{s,t})+\mathbb{W}_1\!\left(\delta_{\textbf{0}}P^{x_1}_{s,t}, \delta_{\textbf{0}}P^{x_2}_{s,t}\right)\\
\le &C e^{-\beta(t-s) }(1+|x_1|+|x_2|)+C|x_1-x_2|.
\end{align*}
Taking $s\to -\infty$, it is easy to see \eref{F3.3} holds. The proof is complete.
\end{proof}

Next, we recall the following definition, see e.g. \cite{DR2006, DR2008}.

\begin{definition}
A class of measures $\{\nu^x_t\}_{t\in\RR}$ is called an evolution system of measures for $\{P^x_{s,t}\}_{t\ge s}$, if
\begin{align}
\int_{\RR^m}P^x_{s,t} \varphi(y)\nu^x_s(dy)=\int_{\RR^m}\varphi(y)\nu^x_t(dy),\quad s\le  t,\varphi\in C_b(\RR^m).\label{ESM}
\end{align}
\end{definition}

We have the following statement.
\begin{proposition} \label{Pr2.5} Suppose that  Assumption {\rm\ref{A2}} holds. Then the class of measures $\{\mu^x_{t}\}_{t\in\RR}$ given in Lemma $\ref{Le3.3}$ is an evolution system of measures for $\{P^x_{s,t}\}_{t\ge s}$. Moreover, for any Lipschitz continuous function $\phi$ on $\RR^m$, $t\ge s$ and $x\in \RR^m$,
\begin{align}
\left|P^x_{s,t} \phi(y)-\int_{\RR^m}\phi(z)\,\mu^x_t(dz)\right|\le  C {\rm Lip(\phi)} (1+|x|+|y|)e^{-\beta(t-s)}.\label{WErodicity}
\end{align}
Furthermore, if $\{\nu^x_{t}\}_{t\in\RR}$ is another evolution system of measures for $\{P^x_{s,t}\}_{t\ge s}$ that satisfies
$$
\sup_{t\in\RR}\int_{\RR^m}|y|\nu^x_t(dy)<\infty,\quad   x\in\RR^n,
$$
then $\nu^x_t=\mu^x_t$ for all $t\in\RR$ and $x\in\RR^n$.
\end{proposition}
\begin{proof} The proof is essentially taken from that of \cite[Propositions 2.3 and 2.5]{DR2006}, see also the proof of  \cite[Proposition 2.6]{SWX2024}. We present the proof here for the sake of the completeness.

(i) According to \eqref{F3.2}, for any $ \varphi \in C_b(\mathbb{R}^{m}) $,
\begin{align}
\lim_{s \to -\infty} P^x_{s,t}\varphi (y)= \int_{\mathbb{R}^{m}}\varphi (y)\mu^x_{t}(dy).\label{Pro1}
\end{align}
By \eqref{FR1}, $P^x_{t,u}\varphi\in C_b(\RR^m)$. Note that
$$P^x_{s,u} P^x_{u,t}\varphi (y) =P^{x}_{s,t} \varphi (y),\quad s\le  u\le  t.$$
Then, letting $s\to -\infty$ in both sides of the equality above, we obtain from \eref{Pro1} that
$$\int_{\mathbb{R}^{m}}P^x_{u,t}\varphi(y)\mu^x_u(dy)=\int_{\mathbb{R}^{m}}\varphi(y)\mu^x_t(dy).$$
Therefore,
$\{\mu^x_{t}\}_{t\in\RR}$ given in Lemma $\ref{Le3.3}$ is an evolution system of measures for $\{P^x_{s,t}\}_{t\ge s}$.

For any Lipschitz continuous function $\phi$ on $\RR^m$, by \eref{F3.2} again, \begin{align*}
\left|P^x_{s,t} \phi(y)\!-\!\int_{\RR^m}\phi(z)\,\mu^x_t(dz)\right|
\le   {\rm Lip}(\phi)\mathbb{W}_1(\delta_{y}P^x_{s,t}, \mu^x_t)
\le C {\rm Lip}(\phi)(1\!+\!|x|\!+\!|y|)e^{-\beta (t-s)}.
\end{align*}
Hence, \eref{WErodicity} holds.

(ii) Let $\{\nu^x_{t}\}_{t\in\RR}$ be another evolution system of measures for
$\{P^x_{s,t}\}_{t\ge s}$. Then
$$ \int_{\mathbb{R}^{m}}P^x_{s,t} \varphi (y)\nu^x _s(dy)=\int_{\mathbb{R}^{m}}\varphi (y)\nu^x_t(dy),\quad\varphi \in C_b^1(\mathbb{R}^{m}). $$
In order to prove the uniqueness, it is sufficient to verify
\begin{align}
\lim_{s \to -\infty} \int_{\mathbb{R}^{m}}P^x_{s,t} \varphi (y)\nu^x_s(dy)=\int_{\mathbb{R}^{m}}\varphi (y)\mu^x_t(dx).\label{SufC}
\end{align}
Then it follows that $\mu^x_t=\nu^x_t$ by the arbitrariness of $\varphi$. In fact, note that
$$\int_{\mathbb{R}^{m}}P^x_{s,t} \varphi (y)\nu^x_s(dy)=\int_{\mathbb{R}^{m}}\left[P^x_{s,t} \varphi (y)-\int_{\mathbb{R}^{m}}\varphi (z) \mu^x_t(dz)\right ]\nu^x_s(dy) +\int_{\mathbb{R}^{m}}\varphi (z)\mu^x_t(dz).$$
On the other hand, according to \eref{WErodicity},
\begin{align*}
 \left |\int_{\mathbb{R}^{m}}\left[P^x_{s,t} \varphi (y)-\int_{\mathbb{R}^{n}}\varphi (z) \mu^x_t(dz)\right ]\nu^x_s(dy) \right|
\le  &C\int_{\mathbb{R}^{m}}(1+|x|+|y|)\nu^x_s(dy)e^{-\beta(t-s)}\\
\le  &C(1+|x|)e^{-\beta (t-s)} .
\end{align*}
Hence, letting $s \to -\infty$, we get that \eref{SufC} holds. The proof is complete.
\end{proof}

\subsection{A priori estimates of $(X_t^\varepsilon,Y_t^\varepsilon)_{t\ge0}$}
Let $(X_t^\varepsilon,Y_t^\varepsilon)_{t\ge0}$ be the unique strong solution to the SDE \eref{Equation}.

\begin{lemma} \label{PMY}
Suppose that Assumptions {\rm\ref{A1}} and {\rm\ref{A2}} hold. For any $p\ge 4$ and  $T>0$, there exists a constant $C_{p,T}>0$ such that for any $\varepsilon>0$, $x\in\RR^{n}$ and $ y\in \RR^{m}$,
\begin{align}
\mathbb{E}\left(\sup_{t\in [0, T]}|X_{t}^{\varepsilon}|^p\right)\le  C_{p,T}(1+|x|^p+|y|^p)\label{X}
\end{align} and
\begin{align}
\sup_{t\in [0,T]}\mathbb{E}|Y_{t}^{\varepsilon}|^p\le  C_{p,T}(1+|x|^p+|y|^p).\label{Y}
\end{align}
Moreover, for any $h>0$ with $0\le  t<t+h\le  T$,
\begin{align}
\mathbb{E}\left(\sup_{t\le  s\le  t+h}|X_{s}^{\varepsilon}\!-\!X_{t}^{\varepsilon}|^{4}\right)
\le  C_{T}\left(1+|x|^{4}+|y|^{4}\right) h^{2}.\label{IncrE}
\end{align}
\end{lemma}

\begin{proof}
(i) Note that
\begin{align*}
Y^{\varepsilon}_{t\varepsilon}= y+\frac{1}{\varepsilon}\int^{t\varepsilon}_0 f(r/\varepsilon,X^{\varepsilon}_{r},Y^{\varepsilon}_r)dr+\frac{1}{\sqrt{\varepsilon}}\int^{t\varepsilon}_0 dW^2_r
=y+\int^{t}_0 f(r,X^{\varepsilon}_{r\varepsilon},Y^{\varepsilon}_{r\varepsilon})dr+\int^{t}_0 d\hat W^2_r,
\end{align*}
where $\hat{W}^2_t:=\frac{1}{\sqrt{\varepsilon}}W^2_{t\varepsilon}$. Denote $\tilde{X}^{\varepsilon}_t:=X^{\varepsilon}_{t\varepsilon}$. Then, by the construction of $(Y^{0,\tilde{X}^{\varepsilon},y}_t)_{t\ge 0}$ and the weak uniqueness of the SDE \eref{GenFRZ}, it is easy to see that $(Y^{\varepsilon}_{t\varepsilon})_{t\ge 0}$ and $(Y^{0,\tilde{X}^{\varepsilon}, y}_{t})_{t\ge 0}$ have the same law.  As a consequence, $(Y^{\varepsilon}_{t})_{t\ge 0}$ has the same law as $(Y^{0,\tilde{X}^{\varepsilon}, y}_{t/\varepsilon})_{t\ge 0}$.
Then, following the proof of \eref{uniformEY}, we have for any $p\ge  4$,
\begin{align}
\EE |Y^{\varepsilon}_t|^p\!=\!\EE|Y^{0,\tilde{X}^{\varepsilon},y}_{t/\varepsilon}|^p\!\le\!  C\left(1\!+\!\sup_{s\in[0, t/\varepsilon]}\EE|\tilde{X}^{\varepsilon}_s|^p\!+\!|y|^p\right)\!\le\!  C\left(1\!+\!\sup_{s\in[0,t]}\EE|X^{\varepsilon}_s|^p\!+\!|y|^p\right).\label{EY}
\end{align}

Using Burkholder-Davis-Gundy's inequality, \eref{ConA12} and \eref{EY}, we get for any $t\in [0,T]$,
\begin{align*}
\mathbb{E}\left(\sup_{s\in [0,t]}|X^{\varepsilon}_t|^p\right)\le  &C_{p,T}(1+|x|^p)+C_{p,T}\int^{t}_{0}\EE|X^{\varepsilon}_s|^p ds+C_{p,T}\int^{t}_{0}\EE|Y^{\varepsilon}_s|^p ds\\
\le  &C_{p,T}(1+|x|^p+|y|^p)+C_{p,T}\int^{t}_{0}\EE\left(\sup_{r\in [0,s]}|X^{\varepsilon}_r|^p \right)ds,
\end{align*}
which along with Gronwall's inequality  yields
\begin{align*}
\mathbb{E}\left(\sup_{s\in [0,t]}|X^{\varepsilon}_t|^p\right)\le  C_{p,T}(1+|x|^p+|y|^p),\quad t\in [0,T].
\end{align*}
Thus, \eref{X} holds. Combining this with \eqref{EY}, we verify \eref{Y}.

(ii) For any $h>0$ with $0\le  t<t+h\le  T$, we have
\begin{align*}
\mathbb{E} \left(\sup_{t\le  s\le  t+h}|X_{s}^{\varepsilon}\!-\!X_{t}^{\varepsilon}|^{4}\right)
\le &C\EE\Big[\sup_{t\le  s\le  t+h}\!\Big|\int_t^{s}b(r,X_{r}^{\varepsilon},Y^{\varepsilon}_{r})dr\Big|^4\Big]\!+\!C\EE\Big[\sup_{t\le  s\le  t+h}\Big|\!\int_t^{s}\sigma(r,X_{r}^{\varepsilon},Y^{\varepsilon}_{r})dW^1_r\Big|^4\Big]\nonumber\\
\le &C h^4\sup_{r\in[0,T]}\mathbb{E}|b(r,X_{r}^{\varepsilon},Y^{\varepsilon}_{r})|^{4}
+C\mathbb{E}\Big(\int_t^{t+h} \|\sigma(r,X_{r}^{\varepsilon},Y^{\varepsilon}_{r})\|^2 dr\Big)^{2}\nonumber\\
\le &C_{T} h^{4}\big(1+ \sup_{r\in[0,T]}\EE|X_{r}^{\varepsilon}|^{4}+\sup_{r\in[0,T]}\EE|Y_{r}^{\varepsilon}|^{4}\big)+C_T h^{2}\big(1+ \sup_{r\in[0,T]}\EE|X_{r}^{\varepsilon}|^{4}\big)\nonumber\\
\le & C_{T}\left(1+|x|^{4}+|y|^{4}\right) h^{2},
\end{align*}
where in the third inequality we used \eqref{ConA12}. The proof of \eref{IncrE} is complete.
\end{proof}

\subsection{An auxiliary process of the fast component}

Inspired by the proof of \cite[Theorem 1]{K1968},
we divide $[0,T)$ into intervals with size $\delta$, where $\delta$ depends on $\varepsilon$ and will be fixed later. Fix each time interval $[k\delta,(k+1)\delta\wedge T)$, where $k=0,1,\ldots, [T/\delta]$, and $[T/\delta]$ is the integer part of $T/\delta$. We construct an auxiliary process $\hat{Y}_{t}^{\varepsilon}$ as follows:
\begin{align*}
\hat{Y}_{t}^{\varepsilon}=Y_{k\delta}^{\varepsilon}+\frac{1}{\varepsilon}\int_{k\delta}^{t}
f(s/\varepsilon, X_{k\delta}^{\varepsilon},\hat{Y}_{s}^{\varepsilon})ds+\frac{1}{\sqrt{\varepsilon}}\int_{k\delta}^{t}dW^{2}_s,\quad t\in [k\delta,[(k+1)\delta]\wedge T).
\end{align*}

The construction of $(\hat{Y}_{t}^{\varepsilon})_{t\ge0}$ and the proof of Lemma \ref{PMY} yield the following estimate.
\begin{lemma} \label{MDY}  Suppose that Assumption {\rm\ref{A2}} holds.
For any $p\ge 4$ and $T>0$, there exists a constant $C_{p,T}>0$ such that for any $\varepsilon>0$, $x\in\RR^{n}$ and $ y\in \RR^{m}$,
\begin{align*}
\sup_{t\in [0, T]}\mathbb{E}|\hat Y_{t}^{\varepsilon}|^{p}\le  C_{p,T}(1+|x|^{p}+|y|^{p}).
\end{align*}
\end{lemma}

Next, we turn to evaluate the difference between $(Y_{t}^{\varepsilon})_{t\ge0}$ and $(\hat{Y}_{t}^{\varepsilon})_{t\ge0}$. We shall mention that the partially dissipative condition of the coefficient $f(t,x,y)$ prevents us from demonstrating an estimate like $\EE|Y_{t}^{\varepsilon}-\hat{Y}_{t}^{\varepsilon}|^2$ as done in the uniformly dissipative setting; see e.g. \cite[Lemma 3.4]{LRSX2020}. The following statement plays an important role in the proofs of our main results.

\begin{lemma} \label{DEY}  Suppose that Assumptions {\rm\ref{A1}} and {\rm\ref{A2}} hold.  Let $F:\RR_{+}\times \RR^n\times\RR^m\rightarrow \RR$ satisfy that for any $t\ge0$, $x_1,x_2\in \RR^n$ and $y_1,y_2\in \RR^m$,
$$
|F(t,x_1,y_1)-F(t,x_2,y_2)|\le  C(|x_1-x_2|+|y_1-y_2|).
$$
Then, for any $T>0$, there is a constant $C_T>0$ such that for any $0\le  t_0<t\le  T$ and $Z_{t_0}\in\mathscr{F}_{t_0}$ with $\EE|Z_{t_0}|^2<\infty$,
\begin{align*}
\left|\EE\left[\int^t_{t_0}\left(F(s/\varepsilon,X^{\varepsilon}_s, Y^{\varepsilon}_s) -F(s/\varepsilon,X^{\varepsilon}_s, \hat Y^{\varepsilon}_s)\right) Z_{t_0}ds\right]\right|
\le  C_{T}(1+|x|+|y|) \left(\EE|Z_{t_0}|^2\right)^{1/2}\delta^{1/2}.
\end{align*}
\end{lemma}

\begin{proof}
For any $\varepsilon, s>0$, $\mathscr{F}_s$-measurable $\RR^n$-valued random variable $X$ and $\RR^m$-valued random variable $Y$, consider the following SDE
\begin{align*}
\tilde{Y}^{\varepsilon,s,X,Y}_t=Y+\frac{1}{\varepsilon}\int^t_s f(r/\varepsilon,X,\tilde{Y}^{\varepsilon,s,X,Y}_r)dr+\frac{1}{\sqrt{\varepsilon}}\int^t_s dW^2_r,\quad t\ge s.
\end{align*}
Then, for $k=0,1,\ldots, [T/\delta]$ and $s\in[k\delta,(k+1)\delta]$,
$\hat{Y}_{s}^{\varepsilon}=\tilde Y^{\varepsilon,k\delta,X_{k\delta}^{\varepsilon},Y_{k\delta}^{\varepsilon}}_s.$
On the other hand, let $(X^{\varepsilon,s, X',Y'}_{t}, Y^{\varepsilon,s, X',Y'}_{t})_{t\ge s}$ be the solution of the SDE
\begin{equation}\left\{\begin{array}{l}\label{InH-Equation}
\displaystyle
d X^{\varepsilon,s, X',Y'}_{t} = b(X^{\varepsilon,s, X',Y'}_{t}, Y^{\varepsilon,s, X',Y'}_{t})dt+\sigma(X^{\varepsilon,s, X',Y'}_{t})d W^1_t,\vspace{2mm}\\
\displaystyle d Y^{\varepsilon,s, X',Y'}_{t} =\varepsilon^{-1}f(t/\varepsilon,X^{\varepsilon,s, X',Y'}_{t},Y^{\varepsilon,s, X',Y'}_{t})dt+\varepsilon^{-1/2}d W^{2}_t
\end{array}\right.
\end{equation}
with $X^{\varepsilon,s, X',Y'}_{s}=X'\in\mathscr{F}_s$ and $Y^{\varepsilon,s, X',Y'}_{s}=Y'\in\mathscr{F}_s$. Then, for $k=0,1,\ldots, [T/\delta]$ and $s\ge  k\delta$,
\begin{equation}\label{e:addsde}
X_{s}^{\varepsilon}=X^{\varepsilon,k\delta,X_{k\delta}^{\varepsilon},Y_{k\delta}^{\varepsilon}}_s,\quad Y_{s}^{\varepsilon}=Y^{\varepsilon,k\delta,X_{k\delta}^{\varepsilon},Y_{k\delta}^{\varepsilon}}_s.
\end{equation}

Note that $X_{k\delta}^{\varepsilon}$ and $Y_{k\delta}^{\varepsilon}$  are $\mathscr{F}_{k\delta}$-measurable, and that $(\tilde{Y}^{\varepsilon, k\delta, x', y'}_{s})_{s\ge k\delta}$ and $(Y^{\varepsilon, k\delta, x', y'}_{s})_{s\ge k\delta}$ are independent of $\mathscr{F}_{k\delta}$ for any $(x',y')\in\RR^n\times\RR^m$. Then,
\begin{align*}
\EE&\left[\int^t_{t_0}\left(F(s/\varepsilon,X^{\varepsilon}_{s(\delta)}, Y^{\varepsilon}_s) -F(s/\varepsilon,X^{\varepsilon}_{s(\delta)}, \hat Y^{\varepsilon}_s)\right) Z_{t_0}ds\right]\\
 =&\sum^{[T/\delta]-1}_{k=0}\int^{[((k+1)\delta)\vee t_0]\wedge t}_{[(k\delta)\vee t_0]\wedge t}\EE\left[\left(F(s/\varepsilon,X^{\varepsilon}_{k\delta}, Y_s^{\varepsilon,k\delta,X_{k\delta}^{\varepsilon},Y_{k\delta}^{\varepsilon}}) -F(s/\varepsilon,X^{\varepsilon}_{k\delta}, \tilde Y^{\varepsilon,k\delta,X_{k\delta}^{\varepsilon},Y_{k\delta}^{\varepsilon}}_s)\right) Z_{t_0}\right]ds\\
=&\sum^{[T/\delta]-1}_{k=0}\int^{[((k+1)\delta)\vee t_0]\wedge t}_{[(k\delta)\vee t_0]\wedge t}\EE\Big[Z_{t_0}\EE\left(F(s/\varepsilon,x', Y^{\varepsilon,k\delta,x',y'}_s) -F(s/\varepsilon,x', \tilde Y^{\varepsilon,k\delta,x',y'}_s)\right)\Big|_{x'=X^{\varepsilon}_{k\delta}, y'=Y^{\varepsilon}_{k\delta}} \Big]ds,
\end{align*}
where $s(\delta)=[{s}/{\delta}]\delta$.

Furthermore, according to \eref{InH-Equation},
\begin{equation}\label{E2.12.1}\begin{split}
Y^{\varepsilon,k\delta,x',y'}_{\varepsilon s}=&y'+\frac{1}{\varepsilon}\int^{\varepsilon s}_{k\delta} f(r/\varepsilon,X^{\varepsilon,k\delta,x',y'}_r,Y^{\varepsilon,k\delta,x',y'}_{r})dr+\frac{1}{\sqrt{\varepsilon}}\int^{\varepsilon s}_{k\delta} dW^{2}_r\\
=&y'+\int^{s}_{\frac{k\delta}{\varepsilon}} f(r, X^{\varepsilon,k\delta,x',y'}_{r\varepsilon},Y^{\varepsilon,k\delta,x',y'}_{r\varepsilon})dr+\int^{s}_{\frac{k\delta}{\varepsilon}} d\hat{W}^{2}_r.\end{split}
\end{equation}
Denote $\tilde X^{\varepsilon,k\delta, x',y'}_r:=X^{\varepsilon,k\delta,x',y'}_{r\varepsilon}$. Recall that  $(Y_{s}^{\frac{k\delta}{\varepsilon},\tilde X^{\varepsilon,k\delta, x',y'}, y'})_{s\ge k\delta/\varepsilon}$ is the unique solution to the SDE \eqref{GenFRZ}; that is,
\begin{align}
Y_{s}^{\frac{k\delta}{\varepsilon},\tilde X^{\varepsilon,k\delta, x',y'}, y'}=& y'
+\int_{\frac{k\delta}{\varepsilon}}^{s}f(r, \tilde X^{\varepsilon,k\delta, x',y'}_r, Y_{r}^{\frac{k\delta}{\varepsilon}, \tilde X^{\varepsilon,k\delta, x',y'},y'})dr+\int_{\frac{k\delta}{\varepsilon}}^{s}d\hat{W}^2_r.  \label{E2.12.2}
\end{align}
The uniqueness of the weak solutions to the SDEs (\ref{E2.12.1}) and (\ref{E2.12.2}) implies
that the distribution of $\big(Y^{\varepsilon, k\delta, x',y'}_{s\varepsilon}\big)_{\frac{k\delta}{\varepsilon}\le  s\le  \frac{(k+1)\delta}{\varepsilon}}$
coincides with that of $\big(Y_{s}^{\frac{k\delta}{\varepsilon},\tilde X^{\varepsilon,k\delta, x',y'}, y'}\big)_{\frac{k\delta}{\varepsilon}\le  s\le  \frac{(k+1)\delta}{\varepsilon}}$. Similarly, the distribution of $\big(\tilde Y^{\varepsilon, k\delta, x',y'}_{s\varepsilon}\big)_{\frac{k\delta}{\varepsilon}\le  s\le  \frac{(k+1)\delta}{\varepsilon}}$
coincides with that of $\big(Y_{s}^{\frac{k\delta}{\varepsilon},x', y'}\big)_{\frac{k\delta}{\varepsilon}\le  s\le  \frac{(k+1)\delta}{\varepsilon}}$.
As a consequence, by the Lipschitz continuity of $F(x',\cdot)$ and \eref{REGenFRZ}, we have
\begin{align*}
&\left|\EE F(s/\varepsilon,x', Y^{\varepsilon,k\delta,x',y'}_s) -\EE F(s/\varepsilon,x', \tilde Y^{\varepsilon,k\delta,x',y'}_s)\right| \\
 &=\Big|\EE F(s/\varepsilon,x', Y^{\frac{k\delta}{\varepsilon},\tilde X^{\varepsilon,k\delta, x',y'},y'}_{s/\varepsilon}) -\EE F(s/\varepsilon,x', Y^{\frac{k\delta}{\varepsilon},x',y'}_{s/\varepsilon})\Big| \\
& \le C\mathbb{W}_1\!\Big(\mathscr{L}_{Y^{\frac{k\delta}{\varepsilon},\tilde X^{\varepsilon,k\delta, x',y'},y'}_{s/\varepsilon}}, \mathscr{L}_{ Y^{\frac{k\delta}{\varepsilon},x',y'}_{s/\varepsilon}}\Big) \\
& \le C\sup_{k\delta\le  r\le  s}\EE|X^{\varepsilon,k\delta, x',y'}_{r}-x'|.
\end{align*}

Hence, according to \eref{IncrE}, we get
\begin{equation}\label{F3.4}\begin{split}
&\left|\EE\left[\int^t_{t_0}\left(F(s/\varepsilon,X^{\varepsilon}_{s(\delta)}, Y^{\varepsilon}_s) -F(s/\varepsilon,X^{\varepsilon}_{s(\delta)}, \hat Y^{\varepsilon}_s)\right) Z_{t_0}ds\right]\right|\\
&\le C_T\sum^{[T/\delta]-1}_{k=0}\int^{[((k+1)\delta)\vee t_0]\wedge t}_{[(k\delta)\vee t_0]\wedge t}\left|\EE\Big[Z_{t_0}\sup_{k\delta\le  r\le  s}
|X^{\varepsilon}_{r}-X^{\varepsilon}_{k\delta}|\Big]\right|ds\\
& \le C_T\sum^{[T/\delta]-1}_{k=0}\int^{[((k+1)\delta)\vee t_0]\wedge t}_{[(k\delta)\vee t_0]\wedge t}\left(\EE|Z_{t_0}|^2\right)^{1/2}\left[\EE\left(\sup_{k\delta\le  r\le  s}
|X^{\varepsilon}_{r}-X^{\varepsilon}_{k\delta}|^2\right)\right]^{1/2}ds\\
&\le C_T(1+|x|+|y|)\left(\EE|Z_{t_0}|^2\right)^{1/2}\delta^{1/2},
\end{split}\end{equation} where in the first inequality we used \eqref{e:addsde}, and the second inequality follows from the Cauchy-Schwarz inequality.
Therefore, by \eref{F3.4} and \eref{IncrE},  we obtain that
\begin{align*}
&\left|\EE\left[\int^t_{t_0}\left(F(s/\varepsilon,X^{\varepsilon}_{s}, Y^{\varepsilon}_s) -F(s/\varepsilon,X^{\varepsilon}_{s}, \hat Y^{\varepsilon}_s)\right) Z_{t_0}ds\right]\right|\\
&\le \left|\EE\left[\int^t_{t_0}\left(F(s/\varepsilon,X^{\varepsilon}_s, Y^{\varepsilon}_s) -F(s/\varepsilon,X^{\varepsilon}_s, \hat Y^{\varepsilon}_s)\right) Z_{t_0}ds\right]\right.\\
&\quad\quad\quad-\left.\EE\left[\int^t_{t_0}\left(F(s/\varepsilon,X^{\varepsilon}_{s(\delta)}, Y^{\varepsilon}_s) -F(s/\varepsilon,X^{\varepsilon}_{s(\delta)}, \hat Y^{\varepsilon}_s)\right) Z_{t_0}ds\right]\right|\\
&\quad+\left|\EE\left[\int^t_{t_0}\left(F(s/\varepsilon,X^{\varepsilon}_{s(\delta)}, Y^{\varepsilon}_s) -F(s/\varepsilon,X^{\varepsilon}_{s(\delta)}, \hat Y^{\varepsilon}_s)\right) Z_{t_0}ds\right]\right|\\
 &\le  C_T\int^{t}_{t_0} \EE\left(|X^{\varepsilon}_s-X^{\varepsilon}_{s(\delta)}||Z_{t_0}|\right)ds+C_{T}(1+|x|+|y|) \left(\EE|Z_{t_0}|^2\right)^{1/2}\delta^{1/2}\\
&\le  C_{T}(1+|x|+|y|)\left(\EE|Z_{t_0}|^2\right)^{1/2}\delta^{1/2}.
\end{align*}
The proof is complete.
\end{proof}

\section{Proof of the strong averaging principle}

In this section, we will give the proof of Theorem \ref{main result 2}. As mentioned in Section \ref{Section1}, we will make use of the averaged SDE \eqref{MAE},
where the drift term involves an evolution system of measures $\{\mu^x_t\}_{t\in \RR}$ that was established in Proposition \ref{Pr2.5}.

\begin{lemma} \label{PMA} Suppose that Assumptions {\rm\ref{A1}} and {\rm\ref{A2}} hold.
For any $\varepsilon>0$ and $x\in \RR^n$, the SDE \eqref{MAE} has a unique strong solution $(\bar{X}^{\varepsilon}_t)_{t\ge0}$. Moreover, for any $p\ge 4$ and $T>0$, there exists a constant $C_{p,T}>0$ such that
\begin{align}
\sup_{\varepsilon\in(0,1]}\mathbb{E}\Big(\sup_{t\in [0, T]}|\bar{X}^{\varepsilon}_{t}|^{p}\Big)\le  C_{p,T}(1+|x|^{p}).\label{EAVE}
\end{align}
\end{lemma}

\begin{proof} According to Assumption {\rm\ref{A1}}, it is sufficient to prove that there is a constant $C>0$ such that for any $t\ge 0$ and $x_1,x_2\in\RR^n$,
\begin{align}\label{barc1}\begin{split}
|\bar b(t,x_{1})-\bar b(t,x_{2})|\le  C|x_{1}-x_{2}|,\quad\quad |\bar{b}(t,x_1)|\le  C(1+|x_1|).
\end{split}
\end{align}
Then the SDE \eqref{MAE} admits a unique strong solution, and \eref{EAVE} holds obviously.

By \eqref{ConA11} and \eqref{F3.3}, for all $t\ge0$ and $x_1,x_2\in \RR^n$,
\begin{align*}
|\bar b(t,x_{1})-\bar b(t,x_{2})|
=&\left|\int b(t,x_{1},y)\mu^{x_1}_t(dy)-\int b(t,x_{2},y)\mu^{x_2}_t(dy)\right|\\
\le &\left|\int b(t,x_{1},y)\mu^{x_1}_t(dy)-\int b(t,x_{2},y)\mu^{x_1}_t(dy)\right|\\
&+\left|\int b(t,x_{2},y)\mu^{x_1}_t(dy)-\int b(t,x_{2},y)\mu^{x_2}_t(dy)\right|\\
\le&C|x_1-x_2|.
\end{align*}
Meanwhile, using \eqref{ConA12} and \eqref{F3.2}, we get that for all $t\ge0$ and $x_1\in \RR^n$,
\begin{align*}
|\bar b(t,x_{1})|
=\left|\int_{\RR^m} b(t,x_{1},y)\mu^{x_1}_t(dy)\right|\le  C\int_{\RR^m} (1+|x_1|+|y|)\mu^{x_1}_t(dy)\le C(1+|x_1|).
\end{align*}
The proof is complete.		
\end{proof}

The following result is a crucial link in establishing the proof of the strong averaging principle.

\begin{lemma} \label{lemma 3.2} Suppose that $\sigma(t,x,y)=\sigma(t,x)$ and Assumptions {\rm\ref{A1}} and {\rm\ref{A2}} hold.
Let $(X_t^\varepsilon,Y_t^\varepsilon)_{t\ge0}$ and $(\bar{X}_t^{\varepsilon})_{t\ge0}$ be the strong solution to the SDEs \eqref{Equation} and \eqref{MAE}, respectively. Then, for any $\varepsilon>0$,  $(x,y)\in\RR^{n}\times\RR^{m}$ and $T>0$,
\begin{align}
\sup_{t\in[0,T]}\mathbb{E}|X_{t}^{\varepsilon}-\bar{X}^{\varepsilon}_{t}|^2\le  C_T(1+|x|^2+|y^2|)\varepsilon^{1/3}. \label{R1}
\end{align}
\end{lemma}

\begin{proof}
Since $\sigma(t,x,y)=\sigma(t,x)$, for any $s\in [0,T]$,
\begin{align*}
X_{s}^{\varepsilon}-\bar{X}^{\varepsilon}_{s}
=&\int_{0}^{s}\left[b(r/\varepsilon,X_{r}^{\varepsilon},Y_{r}^{\varepsilon})-\bar{b}(r/\varepsilon,\bar{X}^{\varepsilon}_{r})\right]dr
+\int_{0}^{s}\left[\sigma(r/\varepsilon,X^{\varepsilon}_{r})-\sigma(r/\varepsilon,\bar{X}^{\varepsilon}_{r})\right]dW^{1}_r\\
=&\int_{0}^s\!\left[b(r/\varepsilon,X_{r}^{\varepsilon},Y_{r}^{\varepsilon})\!-\!b(r/\varepsilon, X_{r}^{\varepsilon},\hat{Y}_{r}^{\varepsilon})\right]dr\!+\!\int_0^s\!\left[b(r/\varepsilon,X_{r}^{\varepsilon},\hat{Y}_{r}^{\varepsilon})\!-\!b(r/\varepsilon, X_{r(\delta)}^{\varepsilon},\hat{Y}_{r}^{\varepsilon})\right]dr\\
&+\int_{0}^{s}\left[b(r/\varepsilon, X_{r(\delta)}^{\varepsilon},\hat{Y}_{r}^{\varepsilon})-\bar{b}(r/\varepsilon,X_{r(\delta)}^{\varepsilon})\right]dr+\int_{0}^{s}\left[\bar{b}(r/\varepsilon, X_{r(\delta)}^{\varepsilon})-\bar{b}( r/\varepsilon,X_{r}^{\varepsilon})\right]dr\\
&+\int_{0}^{s}\left[\bar{b}(r/\varepsilon,X_{r}^{\varepsilon})-\bar{b}(r/\varepsilon,\bar X^{\varepsilon}_{r})\right]dr+\int_{0}^{s}\left[\sigma(r/\varepsilon, X^{\varepsilon}_{r})-\sigma(r/\varepsilon,\bar{X}^{\varepsilon}_{r})\right]dW^{1}_r.
\end{align*}
Thus, for any $t\in [0,T]$,
\begin{align*}
&\sup_{s\in[0,t]}\EE|X_{s}^{\varepsilon}-\bar{X}^{\varepsilon}_{s}|^2\\
 &\le C_T\sup_{s\in [0,t]}\EE\left|\int_{0}^{s}\left[b(r/\varepsilon, X_{r}^{\varepsilon},Y_{r}^{\varepsilon})-b(r/\varepsilon, X_{r}^{\varepsilon},\hat{Y}_{r}^{\varepsilon})\right]dr\right|^2\\
&\quad +C_T\int_{0}^{t}\EE\left|b(r/\varepsilon,X_{r}^{\varepsilon},\hat{Y}_{r}^{\varepsilon})-b(r/\varepsilon, X_{r(\delta)}^{\varepsilon},\hat{Y}_{r}^{\varepsilon})\right|^2dr\nonumber\\
&\quad +C_T\sup_{s\in [0,t]}\EE\left|\int_{0}^{s}\left[b(r/\varepsilon,X_{r(\delta)}^{\varepsilon},\hat{Y}_{r}^{\varepsilon})-\bar{b}( r/\varepsilon,X_{r(\delta)}^{\varepsilon})\right]dr\right|^2\\
&\quad +C_T\!\!\int_{0}^{t}\!\!\EE\Big[\left|\bar{b}(r/\varepsilon, X_{r(\delta)}^{\varepsilon})\!-\!\bar{b}(r/\varepsilon,X_{r}^{\varepsilon})\right|^2\!\!+\!\left|\bar{b}(r/\varepsilon,X_{r}^{\varepsilon})\!-\!\bar{b}(r/\varepsilon,\bar X^{\varepsilon}_{r})\right|^2\!\!+\!\left\|\sigma(r/\varepsilon, X^{\varepsilon}_{r})\!-\!\sigma(r/\varepsilon,\bar{X}^{\varepsilon}_{r})\right\|^2\! \Big]dr.
\end{align*}
Then, by \eqref{ConA11}, \eqref{barc1} and \eqref{IncrE}, we have
\begin{align*}
\sup_{s\in[0,t]}\EE|X_{s}^{\varepsilon}-\bar{X}^{\varepsilon}_{s}|^2\le  &C_T(1+|x|^2+|y^2|)\delta+C_T\int_{0}^{t}\EE|X^{\varepsilon}_{r}-\bar{X}^{\varepsilon}_{r}|^2 dr\\
&+C_T\sup_{s\in [0,t]}\EE\left|\int_{0}^{s}\left[b(r/\varepsilon,X_{r}^{\varepsilon},Y_{r}^{\varepsilon})-b(r/\varepsilon,X_{r}^{\varepsilon},\hat{Y}_{r}^{\varepsilon})\right]dr\right|^2\nonumber\\
&+C_T\sup_{s\in [0,t]}\EE\left|\int_{0}^{s}\left[b(r/\varepsilon,X_{r(\delta)}^{\varepsilon},\hat{Y}_{r}^{\varepsilon})-\bar{b}(r/\varepsilon,X_{r(\delta)}^{\varepsilon})\right]dr\right|^2.
\end{align*}
This along with Gronwall's inequality yields that
\begin{equation}\label{I3.14}\begin{split}
&\sup_{t\in[0,T]}\EE|X_{t}^{\varepsilon}-\bar{X}^{\varepsilon}_{t}|^2\\
 &\le C_T(1+|x|^2+|y^2|)\delta+C_T\sup_{t\in [0,T]}\EE\Big|\int_{0}^{t}\left[b(s/\varepsilon,X_{s}^{\varepsilon},Y_{s}^{\varepsilon})-b(s/\varepsilon, X_{s}^{\varepsilon},\hat{Y}_{s}^{\varepsilon})\right]ds\Big|^2\\
&\quad +C_T\sup_{t\in [0,T]}\EE\Big|\int_{0}^{t}\left[b(s/\varepsilon,X_{s(\delta)}^{\varepsilon},\hat{Y}_{s}^{\varepsilon})-\bar{b}(s/\varepsilon, X_{s(\delta)}^{\varepsilon})\right]ds\Big|^2\\
&=:C_T(1+|x|^2+|y^2|)\delta+I_1(T)+I_2(T).
\end{split}\end{equation}

According to Lemma \ref{DEY}, \eqref{ConA12}, Lemma \ref{PMY} and Lemma \ref{MDY}, we have
\begin{equation}\label{I1}\begin{split}
&I_1(T)\\
&\le   C_T\sup_{t\in [0,T]}
\EE\!\int_{0}^{t}\!\!\int_{r}^{t}\!\Big\langle b(s/\varepsilon,X_{s}^{\varepsilon},Y_{s}^{\varepsilon})\!-\!b(s/\varepsilon,X_{s}^{\varepsilon},\hat{Y}_{s}^{\varepsilon}), b(r/\varepsilon, X_{r}^{\varepsilon},Y_{r}^{\varepsilon})\!-\!b(r/\varepsilon,X_{r}^{\varepsilon},\hat{Y}_{r}^{\varepsilon})\Big\rangle ds dr\\
&\le  C_T\sup_{t\in [0,T]}\!\int_{0}^{t}\!\left|\EE\!\int_{r}^{t}\Big\langle b(s/\varepsilon,X_{s}^{\varepsilon},Y_{s}^{\varepsilon})\!-\!b(s/\varepsilon,X_{s}^{\varepsilon},\hat{Y}_{s}^{\varepsilon}),b(r/\varepsilon,X_{r}^{\varepsilon},Y_{r}^{\varepsilon})\!-\!b(r/\varepsilon,X_{r}^{\varepsilon},\hat{Y}_{r}^{\varepsilon})\Big\rangle ds \right|dr \\
&\le    C_T(1+|x|^2+|y|^2)\delta^{1/2}.
\end{split}\end{equation}

Next, we will estimate $I_2(T)$. It follows from  \eqref{ConA12}, \eqref{barc1},  Lemma \ref{PMY} and Lemma \ref{MDY} that
\begin{align*}
I_2(T)\le  & C_T\sup_{t\in [0,T]}\EE\Big|\sum_{k=0}^{[t/\delta]-1}\!
\int_{k\delta}^{(k+1)\delta}\left[b(s/\varepsilon,X_{k\delta}^{\varepsilon},\hat{Y}_{s}^{\varepsilon})-\bar{b}(s/\varepsilon, X_{k\delta}^{\varepsilon})\right]ds\Big|^2\\
&+C_T\sup_{t\in [0,T]}\EE\Big|
\int_{[t/\delta]\delta}^{t}\left[b(s/\varepsilon,X_{t(\delta)}^{\varepsilon},\hat{Y}_{s}^{\varepsilon})-\bar{b}(s/\varepsilon, X_{t(\delta)}^{\varepsilon})\right]ds\Big|^2\\
\le & C_T[T/\delta]\sum_{k=0}^{[T/\delta]-1}\!
\int_{k\delta}^{(k+1)\delta}\!\int_{r}^{(k+1)\delta}\Psi_{k}(s,r)ds dr+C_T(1+|x|^2+|y|^2)\delta^2,
\end{align*}
where in the last inequality we used the Cauchy-Schwarz inequality and
\begin{align*}
\Psi_{k}(s,r):=\mathbb{E}\left[
\langle b(s/\varepsilon, X_{k\delta}^{\varepsilon},\hat{Y}_{s}^{\varepsilon})-\bar{b}(s/\varepsilon, X_{k\delta}^{\varepsilon}),
b(r/\varepsilon, X_{k\delta}^{\varepsilon},\hat{Y}_{r}^{\varepsilon})-\bar{b}(r/\varepsilon, X_{k\delta}^{\varepsilon})\rangle \right]
\end{align*}
for $k\delta\le  r\le  s\le  (k+1)\delta$. As pointed out in the proof of Lemma \ref{DEY}, we know that
for $k=0,1,\ldots, [T/\delta]$ and $t\in[k\delta,(k+1)\delta]$,
$\hat{Y}_{t}^{\varepsilon}=\tilde Y^{\varepsilon,k\delta,X_{k\delta}^{\varepsilon},Y_{k\delta}^{\varepsilon}}_t,$
which implies that for all $k\delta\le  r\le  s\le  (k+1)\delta$,
\begin{align*}
\Psi_{k}(s,r)=&\mathbb{E}\!\left[
\Big\langle b\big(s/\varepsilon, X_{k\delta}^{\varepsilon},\tilde{Y}^{\varepsilon, k\delta, X_{k\delta}^{\varepsilon},  Y_{k\delta}^{\varepsilon}}_{s}\big)\!-\!\bar{b}(s/\varepsilon, X_{k\delta}^{\varepsilon}), b (r/\varepsilon, X_{k\delta}^{\varepsilon},\tilde{Y}^{\varepsilon, k\delta, X_{k\delta}^{\varepsilon},  Y_{k\delta}^{\varepsilon}}_{r} )\!-\!\bar{b}(r/\varepsilon, X_{k\delta}^{\varepsilon})\Big\rangle \right].
\end{align*}
Note that $X_{k\delta}^{\varepsilon}$ and $ Y_{k\delta}^{\varepsilon}$  are $\mathscr{F}_{k\delta}$-measurable, and that,  for any $(x',y')\in\RR^n\times\RR^m$,  $\{\tilde{Y}^{\varepsilon, k\delta, x', y'}_{s}\}_{s\ge k\delta}$ is independent of $\mathscr{F}_{k\delta}$. Then, for all $k\delta\le  r\le  s\le  (k+1)\delta$,
\begin{align*}
&\Psi_{k}(s,r)\\
&=\EE\!\left\{\mathbb{E}\!\left[
\Big\langle b\big(s/\varepsilon,X_{k\delta}^{\varepsilon},\tilde {Y}^{\varepsilon, k\delta, X_{k\delta}^{\varepsilon},  Y_{k\delta}^{\varepsilon}}_{s}\big)\!-\!\bar{b}(s/\varepsilon, X_{k\delta}^{\varepsilon}), b\big(r/\varepsilon,X_{k\delta}^{\varepsilon},\tilde{Y}^{\varepsilon, k\delta, X_{k\delta}^{\varepsilon},  Y_{k\delta}^{\varepsilon}}_{r}\big)-\bar{b}(r/\varepsilon, X_{k\delta}^{\varepsilon})\Big\rangle |\mathscr{F}_{k\delta}\right]\right\}\\
&=\EE\Big\{\Big[\mathbb{E}
\Big\langle b\big(s/\varepsilon,x',\tilde {Y}^{\varepsilon, k\delta, x', y'}_{s}\big)\!-\!\bar{b}(s/\varepsilon, x'),  b\big(r/\varepsilon, x',\tilde{Y}^{\varepsilon, k\delta,x', y'}_{r}\big)-\bar{b}(r/\varepsilon, x')\Big\rangle \Big]\Big|_{x'=X_{k\delta}^{\varepsilon}, y'=Y_{k\delta}^{\varepsilon}}\Big\}.
\end{align*}
Furthermore, the proof of Lemma \ref{DEY} implies the distribution of $\big(\tilde Y^{\varepsilon, k\delta, x',y'}_{s\varepsilon}\big)_{\frac{k\delta}{\varepsilon}\le  s\le  \frac{(k+1)\delta}{\varepsilon}}$
coincides with that of
 $\big(Y_{s}^{\frac{k\delta}{\varepsilon},x', y'}\big)_{\frac{k\delta}{\varepsilon}\le  s\le  \frac{(k+1)\delta}{\varepsilon}}$. Thus, \eref{WErodicity}, \eqref{barc1} and Lemma \ref{PMY} yield that for all $k\delta\le  r\le  s\le  (k+1)\delta$,
\begin{align*}
&\Psi_{k}(s,r)\\
 &=\EE\Big\{\!\Big[\mathbb{E}
\Big\langle b\big(s/\varepsilon,x',Y^{k\delta/\varepsilon, x', y'}_{s/\varepsilon}\big)\!-\!\bar{b}(s/\varepsilon, x'), b\big(r/\varepsilon, x',Y^{k\delta/\varepsilon,x', y'}_{r/\varepsilon}\big)\!-\!\bar{b}(r/\varepsilon, x')\Big\rangle \Big]\Big|_{x'=X_{k\delta}^{\varepsilon}, y'=Y_{k\delta}^{\varepsilon}}\Big\}\\
 &=\EE\Big\{\!\Big[\mathbb{E}
\Big\langle\! \Big(\EE b\big(s/\varepsilon,\!x'\!,\!Y^{r/\varepsilon, x', z}_{s/\varepsilon}\big)\!-\!\bar{b}(s/\varepsilon, \!x')\big)\Big|_{z=Y^{k\delta/\varepsilon, x', y'}_{r/\varepsilon}}, \!b\big(r/\varepsilon,\!x'\!,\!Y^{k\delta/\varepsilon,x', y'}_{r/\varepsilon}\big)\!-\!\bar{b}(r/\varepsilon, \!x')\Big\rangle \Big]\Big|_{x'=X_{k\delta}^{\varepsilon}, y'=Y_{k\delta}^{\varepsilon}}\Big\}\\
&\le C\EE\Big\{\EE\left(1+|x'|^{2}+ |Y^{k\delta/\varepsilon, x', y'}_{r/\varepsilon}|^{2}\right)\Big|_{x'=X_{k\delta}^{\varepsilon}, y'=Y_{k\delta}^{\varepsilon}}\Big\}e^{-\frac{\beta (s-r)}{\varepsilon}}\\
 &\le C_T\EE\left(1+|X^{\varepsilon}_{k\delta}|^{2}+|Y_{k\delta}^{\varepsilon}|^{2}\right)e^{-\frac{\beta (s-r)}{\varepsilon}}\\
 &\le C_{T}(1+|x|^{2}+|y|^{2})e^{-\frac{\beta (s-r)}{\varepsilon}}.
\end{align*}
Hence,
\begin{equation}\begin{split}
I_2(T)\le &C_{T}(1+|x|^{2}+|y|^{2})[T/\delta]\sum_{k=0}^{[T/\delta]-1}\!\int_{k\delta}^{(k+1)\delta}
\!\int_{r}^{(k+1)\delta}\!e^{-\frac{\beta (s-r)}{\varepsilon}}dsdr+C_T(1+|x|^2+|y|^2)\delta^2   \\
\le &C_{T}(1+|x|^{2}+|y|^{2})\varepsilon/\delta+C_T(1+|x|^2+|y|^2)\delta^2.\label{I2}
\end{split}\end{equation}

Putting \eref{I3.14}, \eref{I1} and \eref{I2} together, we finally obtain that
\begin{align*}
\sup_{t\in[0,T]}\EE|X_{t}^{\varepsilon}-\bar{X}^{\varepsilon}_{t}|^2\le  C_T(1+|x|^2+|y^2|)\left(\delta^{1/2}+\varepsilon/\delta\right).
\end{align*}
Taking $\delta=\varepsilon^{2/3}$, we find that
\begin{align}
\sup_{t\in[0,T]}\EE|X_{t}^{\varepsilon}-\bar{X}^{\varepsilon}_{t}|^2
\le  C_T(1+|x|^2+|y^2|)\varepsilon^{1/3}. \label{FAV}
\end{align}
The proof is complete. \end{proof}

Next, we consider the  averaged SDE \eqref{e:ASDE1} as stated in the statement of the strong averaging principle.
For this, we will study the SDE \eqref{e:IN}, which is closely connected with the long-time behaviors of the fast component $(Y_t^\varepsilon)_{t\ge0}$.
Note that Assumption {\rm\ref{A2}} and condition \eqref{Cf} imply that
\begin{align}
\big\langle \bar f(x_1,y_1)\!-\!\bar f(x_2,y_2),y_1\!-\!y_2\big\rangle
\!\le\!  \begin{cases} C|y_1\!-\!y_2|^2\!+\!C|x_1\!-\!x_2||y_1\!-\!y_2|,\! & |y_1\!-\!y_2|\le r_0,\\
-K|y_1\!-\!y_2|^2\!+\!C|x_1\!-\!x_2||y_1\!-\!y_2|,\! &|y_1\!-\!y_2|>r_0,\end{cases}\label{F6.1}
\end{align}
and for any $x\in\mathbb{R}^n$,
\begin{align*}
|\bar f(x,{\bf0})|\le  C_*(1+|x|).
\end{align*}
In particular, by the same argument as used in the proof of \eqref{A23}, for all $(x,y)\in \mathbb{R}^n\times\mathbb{R}^m$,
\begin{align*}
\langle \bar f(x,y),y\rangle\le  -K_1|y|^2+K_2(1+|x|^2),
\end{align*}
where $0<K_1<K$ and $K_2>0$. Therefore,  the SDE \eqref{e:IN}
admits a unique strong solution $(\bar Y_{t}^{x,y})_{t\ge 0}$, so that the following statements hold:

\begin{itemize}
\item[{\rm(i)}] For any $p\ge4$, there are constants $C_{1,p}, C_{2,p}>0$ such that for all $t\ge0$, $x\in \RR^n$ and $y\in \RR^m$,
\begin{equation}\label{e:IN1}\EE |\bar Y_t^{x,y}|^p\le e^{-C_{1,p}t }|y|^p+C_{2,p}(1+|x|^p).\end{equation}

\item[{\rm (ii)}] For any $t\ge0$, $x_1,x_2\in \RR^n$ and $y_1,y_2\in \RR^m$,
\begin{equation}\label{barErgodicity1}\mathbb{W}_1(\delta_{y_1}\bar{P}^{x_1}_{t}, \delta_{y_2}\bar{P}^{x_2}_{t})\le Ce^{-\beta t}|y_1-y_2|+C|x_1-x_2|,\end{equation} where $\{\bar{P}^{x}_{t}\}_{t\ge 0}$ is the transition semigroup of $(\bar Y_{t}^{x,y})_{t\ge0}$, and $\delta_y \bar P^{x}_{t}$ is the distribution of $\bar Y^{x,y}_t$.

\item[{\rm (iii)}] The process $(\bar Y_{t}^{x,y})_{t\ge 0}$  has a unique invariant measure $\mu^x$, which satisfies that for all $t\ge0$, $x\in \RR^n$ and $y\in \RR^m$,
\begin{equation}\label{barErgodicity2}\mathbb{W}_1(\delta_{y}\bar{P}^{x}_{t}, \mu^x)\le C(1+|x|+|y|)e^{-\beta t} , \quad \int_{\RR^m}|y|\mu^x(dy)\le C(1+|x|),
\end{equation} and that for any $x_1,x_2\in\mathbb{R}^n$, \begin{align}
\mathbb{W}_1(\mu^{x_1}, \mu^{x_2})\le C|x_1-x_2|,\label{Lipmux}
\end{align} where $\beta>0$ is given in Lemma  $\ref{L3.1}$.
\end{itemize}
To obtain all the assertions above, one can follow the proofs of Lemmas \ref{L3.1} and \ref{Le3.3}. We omit the details here.

\begin{lemma}\label{Lemma6.2}
Suppose that Assumptions {\rm\ref{A1}}, {\rm\ref{A2}} and {\rm\ref{A3}} hold.  Let  $\{\mu^x_{t}\}_{t\in\RR}$ be an evolution system of measures  given in Lemma $\ref{Le3.3}$, and $\mu^x$ be the invariant probability measure given above. Then,
\begin{align}
\mathbb{W}_1(\mu^x_t, \mu^x)\le C(1+|x|)\left[e^{-\beta t}+\int^t_0 e^{-\beta (t-u)}\phi_2(u)\,du\right],   \label{SErgodicity}
\end{align}
where $\beta>0$ is defined in Lemma $\ref{L3.1}$.
\end{lemma}
\begin{proof}
We first estimate $\mathbb{W}_{1}(\delta_y P^x_{s,t},\delta_y\bar{P}^x_{t-s})$. For this, we use the notations used in part (ii) of the proof for Lemma \ref{L3.1}, and rewrite the process
$(Y^{s,x,y}_t)_{t\ge s}$ solving the SDE \eqref{FrozenE} as in here, with the solution $(\bar{Z}^{s,x,y}_t)_{t\ge s}$ to the following SDE instead of $(Z_t^{s,x,y})_{t\ge s}$:
\begin{align*}
d\bar{Z}^{s,x,y}_t=&\bar f(x,Z^{s,x,y}_t)\,d t\\
&+\pi_\delta(|Y^{s,x,y}_t-\bar{Z}^{s,x,y}_t|)\left(d W^{2,1}_t-2\frac{(Y^{s,x,y}_t-\bar{Z}^{s,x,y}_t)\langle Y^{s,x,y}_t-\bar{Z}^{s,x,y}_t, dW_t^{2,1}\rangle}{|Y^{s,x,y}_t-\bar{Z}^{s,x,y}_t|^2}\right)\\
&+\sqrt{1-\pi_\delta(|Y^{s,x,y}_t-\bar{Z}^{s,x,y}_t|)}dW_t^{2,2}
\end{align*} with $\bar Z^{s,x,y}_s=y$.
In particular, the processes $(\bar{Y}^{x,y}_{t-s})_{t\ge s}$ and $(\bar{Z}^{s,x,y}_t)_{t\ge s}$ have the same distribution.
	
Let $L_t:= Y^{s,x,y}_t-\bar{Z}^{s,x,y}_t$ for $t\ge s$. We use the function $h$ defined in part (ii) of the proof of Lemma \ref{L3.1}. Following the arguments in
part (ii) of the proof of Lemma \ref{L3.1}, and using Assumption {\rm\ref{A2}} and \eqref{Cf}, we can find that
\begin{align*}
d h(|L_t|)\le  & \Big[ 2\pi_\delta(|L_t|)h''(|L_t|)\!+\!\langle f(t,x,Y^{s,x,y}_t)\!-\!\bar f(x,\bar{Z}^{s,x,y}_t),Y^{s,x,y}_t\!-\!\bar Z^{s,x,y}_t\rangle \frac{h'(|L_t|)}{|L_t|}\Big]dt \!+\!dM_t\\
\le &\Big[ 2\pi_\delta(|L_t|)h''(|L_t|)\!+\!\langle f(t,x,Y^{s,x,y}_t)\!-\!f(t,x,\bar{Z}^{s,x,y}_t),Y^{s,x,y}_t\!-\!\bar Z^{s,x,y}_t\rangle\frac{h'(|L_t|)}{|L_t|}\Big] \,dt \!+\!dM_t\\
&+\left| f(t,x,\bar{Z}^{s,x,y}_t)-\bar f(x,\bar{Z}^{s,x,y}_t)\right| h'(|L_t|) \,dt\\
\le  &-\beta h(|L_t|)\,dt+\left[C\phi_2(t)(1+|x|+|\bar{Z}^{s,x,y}_t|^k)\,+C(\delta)\right]\,dt+dM_t,
\end{align*}
where $(M_t)_{t\ge s}$ is a martingale, and in the last inequality we used $\|h'\|_\infty<\infty$, $\beta>0$ is given in Lemma \ref{L3.1}, and the constant $C(\delta)$ depends on $\delta$ and satisfies that $\lim_{\delta\to 0}C(\delta)=0$.  Since $(\bar{Y}^{x,y}_{t-s})_{t\ge s}$ and $(\bar{Z}^{s,x,y}_t)_{t\ge s}$ have the same distribution,  taking the expectation in both sides above and using \eqref{e:IN1}, we obtain that
$$
d \EE h(|L_t|)\le \left[C\phi_2(t)(1+|x|+|y|^k)+C(\delta)\right]\,dt-\beta\EE h(|L_t|)\,dt,
$$
which in turn implies that
$$
\EE h(|L_t|)\le  C(1+|x|+|y|^k)\int^t_s e^{-\beta (t-u)}\phi_2(u)\,du+C(\delta),\quad t\ge s.
$$
Letting $\delta\to0$,  one has
\begin{align}
\mathbb{W}_{1}(\delta_y P^x_{s,t},\delta_y\bar{P}^x_{t-s}) \le  C(1+|x|+|y|^k)\int^t_s e^{-\beta(t-u)}\phi_2(u)\,du.\label{F6.10}
\end{align}	

According to \eref{F3.2}, \eref{barErgodicity2} and \eref{F6.10}, we finally obtain that
\begin{align*}
\mathbb{W}_{1}(\mu^x_t,\mu^x)\le & \mathbb{W}_{1}(\mu^x_t,\delta_{\textbf{0}} P^x_{0,t})+\mathbb{W}_{1}(\delta_{\textbf{0}} P^x_{0,t},\delta_{\textbf{0}}\bar{P}^x_{t})+\mathbb{W}_{1}(\delta_{\textbf{0}}\bar{P}^x_{t},\mu^x)\\
\le & C(1+|x|)e^{-\beta t}+C(1+|x|)\int^t_0 e^{-\beta(t-u)}\phi_2(u)\,du\\
\le & C(1+|x|)\left[e^{-\beta t}+\int^t_0 e^{-\beta(t-u)}\phi_2(u)\,du\right].
\end{align*}
The proof is complete.
\end{proof}

For the  averaged SDE \eqref{e:ASDE1}, we have the following statement.

\begin{lemma} \label{FianlW} Suppose that Assumptions {\rm\ref{A1}}, {\rm\ref{A2}}, {\rm\ref{A3}} and  {\rm\ref{A4}} hold.
For any $x\in\RR^{n}$, the SDE \eqref{e:ASDE1} has a unique strong solution $(\bar{X}_t)_{t\ge0}$. Moreover, for any $p\ge 2$ and $T>0$, there exists a constant $C_{p,T}>0$ such that
\begin{align}
\mathbb{E}\left(\sup_{t\in [0, T]}|\bar{X}_{t}|^{p}\right)\le  C_{p,T}(1+|x|^{p}),\label{FianlEE}
\end{align}
and, for any $h>0$ with $0\le  t<t+h\le  T$,
\begin{align}
\mathbb{E}|\bar X_{t+h}-\bar X_{t}|^{2}
\le  C_{T}\left(1+|x|^{2}\right) h.\label{IncrEbarX}
\end{align}
\end{lemma}

\begin{proof}
According to \eref{ConA11}, \eref{Cb} and \eref{Csigam}, for all $x_1,x_2\in \RR^n$ and $y_1,y_2\in \RR^m$,
\begin{align}
|\hat b(x_{1},y_1)-\hat b(x_{2},y_2)|+\|\bar\sigma(x_1)-\bar\sigma(x_2)\|\le  C(|x_{1}-x_{2}|+|y_1-y_2|).\label{Lipbsigma}
\end{align}
Then, \eref{Lipmux} and \eref{Lipbsigma} imply that for all $x_1,x_2\in \RR^n$,
\begin{equation}\label{Lipbarb}\begin{split}
|\bar b(x_{1})-\bar b(x_{2})|
=&\left|\int \hat b(x_{1},y)\mu^{x_1}(dy)-\int \hat b(x_{2},y)\mu^{x_2}(dy)\right|\\
\le &\left|\int \hat b(x_{1},y)\mu^{x_1}(dy)\!-\!\!\int \hat b(x_{2},y)\mu^{x_1}(dy)\right|\\
&+\!\left|\int \hat b(x_{2},y)\mu^{x_1}(dy)\!-\!\!\int \hat b(x_{2},y)\mu^{x_2}(dy)\right|\\
\le& C|x_1-x_2|.
\end{split}\end{equation}
Hence, by \eref{Lipbarb} and \eref{Lipbsigma}, the SDE \eref{e:ASDE1} admits a unique strong solution, and the other assertions hold by standard arguments. The proof is complete.
\end{proof}

We also need the simple lemma as follows.

\begin{lemma}\label{Pro4.4}
Let $\phi:\RR_{+}\to \RR_{+}$ be locally integrable and satisfy $\lim_{T\to+\infty}\phi(T)=0$. Then,
$$\lim_{T\to +\infty}\sup_{t\ge0}\frac{1}{T}\int_{t}^{t + T}\phi(s)ds=0.$$
\end{lemma}

\begin{proof}
For any $\epsilon>0$,  $\lim_{T\to+\infty}\phi(T)=0$ implies that there exists $M_{\epsilon}>0$ such that for $s\ge M_{\epsilon}$, $\phi(s)<\epsilon$. Then, for any $T>0$ and $t\ge0$,
\begin{align*}
\int_{t}^{t + T}\phi(s)ds=\int_{t}^{\max\{t,M_{\epsilon}\}}\phi(s)ds+\int_{\max\{t,M_{\epsilon}\}}^{t + T}\phi(s)ds
\le  \int_{0}^{M_{\epsilon}}\phi(s)ds+\epsilon T,
\end{align*}
which implies that
\begin{align*}
\sup_{t\ge 0}\frac{1}{T}\int_{t}^{t + T}\phi(s)ds \le  \frac{1}{T}\int_{0}^{M_{\epsilon}}\phi(s)ds+\epsilon.
\end{align*}
Taking $T\to+\infty$ first, and then $\epsilon\rightarrow 0$, we get the desired result.
\end{proof}

Now, we will present the  proof of our first main result.

\begin{proof}[Proof of Theorem $\ref{main result 2}$]
According to Lemma \ref{Lemma6.2} and Assumption \ref{A3}, for all $t\ge0$, $T>0$ and $x\in \RR^n$,
\begin{equation}\label{F4.19}\begin{split}
\frac{1}{T} \left|\int^{t+T}_t [\bar{b}(s,x)-\bar{b}(x)]ds\right|
=& \frac{1}{T}\left|\int^{t+T}_t \left[\int_{\RR^m}b(s,x,y)\mu^x_s(dy)-\int_{\RR^m}\hat b(x,y)\mu^x(dy)\right]ds\right|\\
\le &\frac{1}{T}\left|\int^{t+T}_t\! \left[\int_{\RR^m}\!\!\!b(s,x,y)\mu^x_s(dy)\!-\!\!\int_{\RR^m}\!\!\! b(s,x,y)\mu^x(dy)\right]ds\right|\\
&+\frac{1}{T}\left|\int^{t+T}_t \!\left[\int_{\RR^m}\!\!\!b(s,x,y)\mu^x(dy)\!-\!\!\int_{\RR^m}\!\!\!\hat b(x,y)\mu^x(dy)\right]ds\right|\\
\le & C(1+|x|)\frac{1}{T}\int^{t+T}_t \left[e^{-\beta s}+\int^s_0 e^{-\beta(s-u)}\phi_2(u)\,du\right]ds\\ &+C\int_{\RR^m} (1+|x|+|y|)\mu^x(dy)\phi_1(T)\\
\le &C(1+|x|)[\phi_1(T)+\tilde{\phi}_2(T)],
\end{split}\end{equation}
where in the last inequality we used \eqref{barErgodicity2} and
\begin{align}
\tilde{\phi}_2(T):=\sup_{t\ge 0}\frac{1}{T}\int^{t+T}_t \left[e^{-\beta s}+\int^s_0 e^{-\beta(s-u)}\phi_2(u)\,du\right]ds.\label{phi5}
\end{align}
Note that, by \cite[Remark 4.2]{SWX2024}, $\lim_{s\to\infty}\!
\int^s_0 \!e^{-\beta(s-u)}\phi_2(u)\,du=0$, which yields that $\lim_{T\to \infty}\tilde{\phi}_2(T)=0$ by Lemma \ref{Pro4.4}.

Recall that $(\bar X_t^\varepsilon)_{t\ge0}$ and $(\bar X_t)_{t\ge0}$ are the unique strong solutions to the SDEs \eqref{MAE} and \eqref{e:ASDE1} respectively.
It holds that
\begin{align*}
\bar X_{t}^{\varepsilon}-\bar{X}_{t}=\int_{0}^{t}\left[\bar{b}(s/\varepsilon,\bar{X}^{\varepsilon}_{s})-\bar{b}(\bar{X}_{s})\right]ds
+\int_{0}^{t}\left[\sigma(s/\varepsilon,\bar{X}^{\varepsilon}_{s})-\bar\sigma(\bar{X}_{s})\right]dW^1_s.
\end{align*}
By \eref{ConA11} and \eref{barc1}, for any $t\in [0,T]$
\begin{align*}
 \EE|\bar X_t^\varepsilon-\bar{X}_{t}|^{2}
\le  &C_T\EE\int_0^t|\bar{b}(s/\varepsilon,\bar{X}^{\varepsilon}_{s})  -  \bar{b}(s/\varepsilon,\bar{X}_{s})|^2\,ds  +  C_T\int_0^t \EE\|\sigma(s/\varepsilon,\bar{X}^{\varepsilon}_s)-\sigma(s/\varepsilon,\bar{X}_s)\|^2\,ds \\
&+C_T\EE\left|\int_0^t\bar{b}(s/\varepsilon,\bar{X}_{s})-\bar{b}(\bar{X}_{s}) \,ds\right|^2+C_T\int_0^t \EE\|\sigma(s/\varepsilon,\bar{X}_s)-\bar\sigma(\bar{X}_s)\|^2\,ds\\
\le & C_T\int^t_0 \EE|\bar X_s^\varepsilon-\bar{X}_{s}|^{2}ds+C_T\EE\left|\int_0^t\bar{b}(s/\varepsilon,\bar{X}_{s})-\bar{b}(\bar{X}_{s}) \,ds\right|^2\\
&+C_T\int_0^t \EE\|\sigma(s/\varepsilon,\bar{X}_s)-\bar\sigma(\bar{X}_s)\|^2\,ds
\end{align*}
Thus, by the Gronwall inequality, Assumption \ref{A1}, \eqref{barc1}, \eqref{IncrEbarX}, \eqref{Lipbsigma} and \eqref{Lipbarb}, we have
\begin{align*}
& \sup_{t\in [0,T]}\EE |\bar{X}_t^\varepsilon-\bar{X}_{t} |^{2}\\
  &\le   C_T\sup_{t\in [0,T]}\EE\Big|\int_0^t\![\bar{b}(s/\varepsilon,\bar{X}_{s})-\bar{b}(\bar{X}_{s})]\,ds\Big|^2+C_T\EE\int_0^T\|\sigma(s/\varepsilon,\bar{X}_s)-\bar\sigma(\bar{X}_s)\|^2\,ds\\
 & \le  C_T\!\!\sup_{t\in [0,T]}\EE\Big|\!\int_0^t\!\![\bar{b}(s/\varepsilon,\bar{X}_{s(\delta)})\!-\!\bar{b}(\bar{X}_{s(\delta)})]\,ds\Big|^2\!\!+\!C_T\int_0^t \!\EE\|\sigma(s/\varepsilon,\bar{X}_{s(\delta)})\!-\!\bar\sigma(\bar{X}_{s(\delta)})\|^2\,ds\\
&\quad+C_T\EE\int_0^T|\bar{b}(s/\varepsilon,\bar{X}_{s})-\bar{b}(s/\varepsilon,\bar{X}_{s(\delta)})|^2\,ds+C_T\EE\int_0^T|\bar{b}(\bar{X}_{s})-\bar{b}(\bar{X}_{s(\delta)})|^2\,ds\\
&\quad+C_T\int_0^T \EE\|\sigma(s/\varepsilon,\bar{X}_{s})-\sigma(s/\varepsilon,\bar{X}_{s(\delta)})\|^2\,ds+C_T\int_0^T \EE\|\bar\sigma(\bar{X}_{s})-\bar\sigma(\bar{X}_{s(\delta)})\|^2\,ds\\
 & \le C_T\!\!\sup_{t\in [0,T]}\EE\Big|\!\int_0^t\!\![\bar{b}(s/\varepsilon,\bar{X}_{s(\delta)})\!-\!\bar{b}(\bar{X}_{s(\delta)})]\,ds\Big|^2\!+\!C_T\!\int_0^T \!\EE\|\sigma(s/\varepsilon,\bar{X}_{s(\delta)})-\bar\sigma(\bar{X}_{s(\delta)})\|^2\,ds\\
&\quad+C_T(1+|x|^2)\delta\\
& =:J^{\varepsilon}_1(T)+J^{\varepsilon}_2(T)+C_T(1+|x|^2)\delta.
\end{align*}

For $J^{\varepsilon}_1(T)$, by \eref{F4.19} and \eqref{FianlEE}, we have
\begin{align}\label{E:add1}
J^{\varepsilon}_1(T)\le & C_T\sup_{t\in [0,T]}\EE\left|\sum^{[t/\delta]-1}_{k=0}    \int^{(k+1)\delta}_{k\delta}    [\bar{b}(s/\varepsilon,\bar{X}_{k\delta})  -  \bar{b}(\bar{X}_{k\delta})]  \,ds
\right|^2\nonumber\\
& +  C_T\sup_{t\in [0,T]}\EE\left|\int^{t}_{[t/\delta]\delta}  [\bar{b}(s/\varepsilon,\bar{X}_{k\delta})  -  \bar{b}(\bar{X}_{k\delta})]  \,ds\right|^2\nonumber\\
\le & C_T\sup_{t\in [0,T]}\EE\left|\sum^{[t/\delta]-1}_{k=0}\delta\left(\frac{\varepsilon}{\delta}\int^{\frac{k\delta}{\varepsilon}+\frac{\delta}
{\varepsilon}}_{\frac{k\delta}{\varepsilon}}\bar{b}(s,\bar{X}_{k\delta})\,ds-\bar{b}(\bar{X}_{k\delta})\right)\right|^2+C_T(1+|x|^2)\delta^2\nonumber\\
\le & C_T\left[1+\EE\left(\sup_{t\in [0,T]}|\bar X_t|^2\right)\right]\left(\phi_1(\delta/\varepsilon)+\tilde{\phi}_2(\delta/\varepsilon)\right)^2+C_T(1+|x|^2)\delta^2\nonumber\\
\le & C_T (1+|x|^2 )\left(\phi_1(\delta/\varepsilon)+\tilde{\phi}_2(\delta/\varepsilon)\right)^2+C_T(1+|x|^2)\delta^2.
\end{align}

Following similar arguments as above and using \eref{Csigam} in Assumption {\rm\ref{A4}}, we can obtain
\begin{align*}
J^{\varepsilon}_2(T)\le  C_T (1+|x|^2 )\phi_3(\delta/\varepsilon)+C_T(1+|x|^2)\delta^2.
\end{align*}
Thus, by taking $\delta=\varepsilon^{1/2}$, we immediately get
 \begin{align*}
\sup_{t\in [0,T]}\EE|\bar{X}_t^\varepsilon-\bar{X}_{t}|^{2}\le C_T (1+|x|^2 )\left[(\phi_1(1/\sqrt{\varepsilon})+\tilde{\phi}_2(1/\sqrt{\varepsilon}))^2+\phi_3(1/\sqrt{\varepsilon})+\varepsilon\right].
\end{align*}
This, together with \eref{R1}, gives us that
\begin{equation}\label{e:RMK}\begin{split}
\sup_{t\in [0,T]}\EE|X_t^\varepsilon\!-\!\bar{X}_{t}|^{2}\le C_T (1\!+\!|x|^2\!+\!|y|^2 )\left[(\phi_1(1/\sqrt{\varepsilon})\!+\!\tilde{\phi}_2(1/\sqrt{\varepsilon}))^2\!+\!\phi_3(1/\sqrt{\varepsilon})\!+\!\varepsilon^{1/3}\right].
\end{split}\end{equation}
Therefore, the desired assertion follows from.
\end{proof}

\begin{remark}
The estimate \eqref{e:RMK} indicates that the strong convergence rate is influenced by how fast $\phi_i(T)$ converges to $0$ as $T\to \infty$ for all $i=1,2,3$.
We shall note that the rate given in \eqref{e:RMK} is not optimal.
\end{remark}

\section{Proof of the weak averaging principle}

This section will give the proof of Theorem $\ref{main result 3}$.  We first consider the tightness of $\{X^{\varepsilon}\}_{\varepsilon\in (0,1]}$ in $C([0,T];\RR^{n})$, then we identify the limiting process by the martingale problem approach.

\begin{proposition}\label{pro4.6}
 Suppose that Assumptions {\rm\ref{A1}} and {\rm\ref{A2}} hold. Then, for $T>0$, the distribution of
 $\{X^\varepsilon\}_{\varepsilon\in(0,1]}$ is tight in $C([0,T];\RR^{n})$.
\end{proposition}

\begin{proof}
In order to prove that  $\{X^\varepsilon\}_{\varepsilon\in(0,1]}$ is tight in $C([0,T];\RR^{n})$, we will apply \cite[Theorem 7.3]{B1999}. For this, we need check the following two conditions.

(i) For any $\eta>0$, there exist constants $\tilde K,\varepsilon_0>0$ such that
\begin{equation}\label{T1}
\sup_{0<\varepsilon<\varepsilon_0}\mathbb{P}\big(|X^\varepsilon_0|\ge \tilde K\big)\le \eta.
\end{equation}

(ii) For any $\theta,\eta>0$, there exist constants $\varepsilon_0, h>0$ such that
\begin{equation}\label{T2}
\sup_{0<\varepsilon<\varepsilon_0}\mathbb{P}\Big(\sup_{t_1,t_2\in[0,T] \hbox{ with } |t_1-t_2|<h}|X^\varepsilon_{t_1}-X^\varepsilon_{t_2}|\ge \theta\Big)\le  \eta.
\end{equation}

Since $X^\varepsilon_0=x$, \eref{T1} holds obviously. On the other hand, according to  \cite[(7.12), p.\ 83]{B1999}, a sufficient condition to guarantee \eref{T2} holds is that: for any constant $\theta,\eta>0$, there exist constants $\varepsilon_0,h$ such that for any $0\le  t\le  t+h\le  T$,
\begin{equation}\label{T3}
\sup_{0<\varepsilon<\varepsilon_0}h^{-1}\mathbb{P}\Big(\sup_{t\le  s\le  t+h}|X_{s}^{\varepsilon}-X_{t}^{\varepsilon}|\ge \theta\Big)\le  \eta.
\end{equation}
Indeed, by Chebyshev's inequality and \eref{IncrE}, for any $\theta>0$ and $0\le  t\le  t+h\le  T$,
$$
h^{-1}\PP\big(\sup_{t\le  s\le  t+h}|X_{s}^{\varepsilon}-X_{t}^{\varepsilon}|\ge \theta\big)\le  \frac{C_T\left(1+|x|^{4}+|y|^4\right) h}{\theta^4}.
$$
This shows that \eref{T3}
holds. The proof is complete.
\end{proof}

Next, we study the existence and uniqueness of the strong solution to the averaged SDE \eqref{e:ASDE2}, and establish its a priori estimates.

\begin{lemma} Suppose that Assumptions {\rm\ref{A1}}, {\rm\ref{A2}}, {\rm\ref{A3}} and {\rm\ref{A5}} hold.
For any $x\in\RR^{n}$, the SDE \eqref{e:ASDE2} has a unique strong solution $(X_t)_{t\ge0}$. Moreover, for any $T>0$, there exists a constant $C_{T}>0$ such that
\begin{align}
\mathbb{E}\left(\sup_{t\in [0, T]}|X_{t}|^{2}\right)\le  C_{T}(1+|x|^{2}).\label{FianlEE2}
\end{align}
\end{lemma}

\begin{proof}
It has been proved in \eref{Lipbarb} that the drift term $\bar b$ is Lipschitz continuous. Next, we will check that there exists a constant $C>0$ such that for all $x_1,x_2\in \RR^n$,
\begin{align}
&\|\Theta(x_1)-\Theta(x_2)\|\le  C(1+|x_1|+|x_2|)|x_1-x_2|,\label{Lipbarsigma}\\
&\|\Theta(x_1)\|\le  C(1+|x_1|).\label{Linearbarsigma}
\end{align}
In fact, by \eref{ConA11} and \eqref{ConA12}, there exists a constant $C>0$ such that for all $t\ge0$, $x_1,x_2\in \RR^n$ and $y_1,y_2\in \RR^m$,
$$
\|(\sigma\sigma^{\ast})(t,x_1,y_1)-(\sigma\sigma^{\ast})(t,x_2,y_2)|\le  C(1+|x_1|+|x_2|)(|x_1-x_2|+|y_1-y_2|).$$
Then, according to Assumption \ref{A5}, for all $x_1,x_2\in \RR^n$ and $y_1,y_2\in \RR^m$,  $i=1,2$,
\begin{align}
\|\Sigma(x_1,y_1)-\Sigma(x_2,y_2)|\le  C(1+|x_1|+|x_2|)(|x_1-x_2|+|y_1-y_2|).\label{LipSigma}
\end{align}
Furthermore, since $\bar\Sigma(x)=\int_{\RR^m}\Sigma(x,y)\mu^x(dy), $
 \eref{LipSigma} and \eref{Lipmux} imply that for all $x_1,x_2\in \RR^n$,
$$
\|\bar\Sigma(x_1)-\bar\Sigma(x_2)\|\le  C(1+|x_1|+|x_2|)|x_1-x_2|.$$
In particular,
$\|\bar\Sigma(x)\|\le  C(1+|x|^2)$ for all $x\in \RR^n$, and so we get
$$
\|\bar\Sigma^{1/2}(x)\|\le  \sqrt{n}\|\bar\Sigma(x)\|^{1/2}\le  C\sqrt{n}(1+|x|).
$$
Thus \eref{Linearbarsigma} holds.

On the other hand, the conditions \eqref{Csigam1} and \eqref{NonD} imply that $\Sigma(x,y)$ is uniformly non-degenerate; that is,
$$
\inf_{x\in\RR^n,y\in\RR^m,z\in \RR^n\backslash\{\textbf{0}\}}\frac{\langle \Sigma(x,y)\cdot z, z\rangle}{|z|^2}>0. $$  Hence,  $\bar\Sigma(x)$ is also uniformly non-degenerate in the sense that
$$
\inf_{x\in\RR^n,z\in \RR^n\backslash\{\textbf{0}\}}\frac{\langle \bar\Sigma(x)\cdot z, z\rangle}{|z|^2}>0. $$

In the following, we will need an elementary fact (see \cite{PW2006}). Let $A$ and $B$ be two $n\times n$ symmetric positive matrixes such that all eigenvalues of $A$ and $B$ are not less than $\gamma>0$. Then, it holds that
$$
\|A^{1/2}-B^{1/2}\|\le  \frac{1}{2\gamma}\|A-B\|.
$$
Combining it with all the estimates above, we get that for all $x_1,x_2\in \RR^n$,
$$
\|\bar{\Sigma}^{1/2}(x_1)-\bar{\Sigma}^{1/2}(x_2)\|\le  C(1+|x_1|+|x_2|)|x_1-x_2|,
$$
which yields that \eref{Lipbarsigma} holds.

Therefore, \eref{Lipbarb}, \eref{Lipbarsigma} and \eref{Linearbarsigma} imply the SDE \eref{e:ASDE2} admits a unique strong solution $(X_t)_{t\ge0}$. Moreover, \eref{FianlEE2} holds obviously. The proof is complete.		
\end{proof}

Finally, we give the  proof of our second main result.

\begin{proof}[Proof of Theorem $\ref{main result 3}$] Fix $T>0$. According to Proposition \ref{pro4.6}, it is sufficient to prove that, for any sequence $\{\varepsilon_k\}_{k\ge1}$ so that $\varepsilon_k\to 0$ as $k\to\infty$, there exists a subsequence (for simplicity we still denote by $\{\varepsilon_{k}\}_{k\ge1}$) such that $X^{\varepsilon_k}$ converges weakly to $X:=(X_t)_{t\ge0}$ in $C([0,T];\RR^{n})$, which is the unique strong solution to the SDE \eref{e:ASDE2}. Without loss of generality, we may assume that $X^{\varepsilon_k}$ converges almost surely to $X$ in $C([0,T];\mathbb{R}^{n})$ by the Skorohod representation theorem.

In the following, we shall apply the martingale problem approach to characterize the limit process $X$.
Let $\Psi_{t_0}(\cdot)$ be a  bounded, continuous and $\sigma\{\varphi_s: \varphi\in C([0,T];\RR^{n}), s\le  t_0\}$-measurable function on $C([0,T];\RR^{n})$.  It suffices to prove that for any $U\in C^3_b(\RR^{n})$, the following assertion holds:
\begin{align}
\EE\left[\left(U(X_t)-U(X_{t_0})-\int^t_{t_0}\bar{\mathcal{L}} U(X_s)ds\right)\Psi_{t_0}(X)\right]=0,\quad t_0\le  t\le  T,\label{F3.33}
\end{align}
where $\bar{\mathcal{L}}$ is the generator of the process $(X_t)_{t\ge0}$ to the SDE \eref{e:ASDE2}, i.e.,
$$
\bar{\mathcal{L}} U(x):=\Big\langle \nabla U(x), \bar{b}(x)\Big\rangle+\frac{1}{2}\text{Tr}\left[\bar{\Sigma}(x)\nabla^2 U(x)\right].$$

Applying It\^{o}'s formula, we have
\begin{align*}
U(X^{\varepsilon_k}_t)=&U(X^{\varepsilon_k}_{t_0})\!+\!\!\int^t_{t_0}\!\langle \nabla U(X^{\varepsilon_k}_s),b(s/\varepsilon_k,X^{\varepsilon_k}_s, Y^{\varepsilon_k}_s)\rangle ds\!+\!\!\int^t_{t_0}\!\langle \nabla U(X^{\varepsilon_k}_s),\sigma(s/\varepsilon_k,X^{\varepsilon_k}_s, Y^{\varepsilon_k}_s)d W_s\rangle\nonumber\\
&+\frac{1}{2}\int^t_{t_0}\text{Tr}\left[(\sigma\sigma^{\ast})(s/\varepsilon_k,X^{\varepsilon_k}_s,Y^{\varepsilon_k}_s)\nabla^2 U(X^{\varepsilon_k}_s)\right]ds.
\end{align*}
Note that, by the Markov property and the martingale property,
$$
\EE\left\{\left[\int^t_{t_0}\langle \nabla U(X^{\varepsilon_k}_s),\sigma(X^{\varepsilon_k}_s, Y^{\varepsilon_k}_s)d W_s\rangle\right]\Psi_{t_0}(X^{\varepsilon_k})\right\}=0.
$$ Since $X^{\varepsilon_k}$ converges almost surely to $X$ in $C([0,T];\mathbb{R}^{n})$, it holds that
\begin{align*}
&\lim_{k\to\infty} \EE\left\{\left[U(X^{\varepsilon_k}_t)-U(X^{\varepsilon_k}_{t_0})\right]\Psi_{t_0}(X^{\varepsilon_k})\right\}
=\EE\left\{\left[U(X_t)-U(X_{t_0})\right]\Psi_{t_0}(X)\right\},\\
&\lim_{k\to \infty} \EE\Big\{\Big[\int^t_{t_0}\langle \nabla U(X^{\varepsilon_k}_s),\bar{b}(X^{\varepsilon_k}_s)\rangle ds\Big]\Psi_{t_0}(X^{\varepsilon_k})\Big\}=\EE\Big\{\Big[\int^t_{t_0}\langle \nabla U(X_s),\bar{b}(X_s)\rangle ds\Big]\Psi_{t_0}(X)\Big\},\\
&\lim_{k\to \infty} \EE\Big\{\Big[\int^t_{t_0}\!\text{Tr}\left(\bar \Sigma(X^{\varepsilon_k}_s) \nabla^2 U(X^{\varepsilon_k}_s)\right)ds\Big]\Psi_{t_0}(X^{\varepsilon_k})\Big\}=\EE\Big\{\Big[\!\int^t_{t_0}\!\text{Tr}\left(\bar\Sigma(X_s) \nabla^2 U(X_s)\right)ds\Big]\Psi_{t_0}(X)\Big\}.
\end{align*}
Hence, in order to obtain \eref{F3.33}, we only need to verify that for all $ t_0\le  t\le  T$,
\begin{align}
&\lim_{k\to \infty} \EE\left\{\left[\int^t_{t_0} \langle \nabla U(X^{\varepsilon_k}_s),b(s/\varepsilon_k,X^{\varepsilon_k}_s, Y^{\varepsilon_k}_s) -\bar{b}(X^{\varepsilon_k}_s)\rangle ds\right]\Psi_{t_0}(X^{\varepsilon_k})\right\}=0,\label{F3.34}\\
&\lim_{k\to \infty} \EE\left\{\left[\int^t_{t_0}\text{Tr}\left[\left((\sigma\sigma^{\ast})(s/\varepsilon_k,X^{\varepsilon_k}_s,Y^{\varepsilon_k}_s)-\bar\Sigma(X^{\varepsilon_k}_s)\right)\nabla^2 U(X^{\varepsilon_k}_s)\right]ds\right]\Psi_{t_0}(X^{\varepsilon_k})\right\}=0.\label{F3.35}
\end{align}

Note that
\begin{align*}
&\left|\EE\left\{\left[\int^t_{t_0}\Big\langle \nabla U(X^{\varepsilon_k}_s),b(s/\varepsilon_k,X^{\varepsilon_k}_s, Y^{\varepsilon_k}_s)-\bar{b}(X^{\varepsilon_k}_s)\Big\rangle ds\right]\Psi_{t_0}(X^{\varepsilon_k})\right\}\right| \\
&\le C\EE\left|\int^t_{t_0}\left(\Big\langle \nabla U(X^{\varepsilon_k}_s),b(s/\varepsilon_k,X^{\varepsilon_k}_s, Y^{\varepsilon_k}_s)\Big\rangle -\Big\langle \nabla U(X^{\varepsilon_k}_s),\bar{b}(s/\varepsilon_k,X^{\varepsilon_k}_s)\Big\rangle\right)ds\right| \\
&\quad +C\EE\left|\int^t_{t_0}\left(\Big\langle \nabla U(X^{\varepsilon_k}_s),\bar{b}(s/\varepsilon_k,X^{\varepsilon_k}_s)\Big\rangle -\Big\langle \nabla U(X^{\varepsilon_k}_s),\bar{b}(X^{\varepsilon_k}_s)\Big\rangle\right)ds\right| \\
&=:CQ^{\varepsilon_k}_1+CQ^{\varepsilon_k}_2.
\end{align*}
Set
$$B(t,x,y):=\Big\langle \nabla U(x),b(t,x, y)-\bar b(t,x)\Big\rangle.$$
Then, for all $t\ge0$, $x_1,x_2\in \RR^n$ and $y_1,y_2\in \RR^m$,
\begin{align}\label{LipB}\begin{split}
&|B(t,x_1,y_1)-B(t,x_2,y_2)|\le  C(1+|x_1|+|y_1|)|x_1-x_2|+C|y_1-y_2|,\\
&|B(t,x_1,y_1)|\le  C(1+|x_1|+|y_1|).
\end{split}
\end{align}
We have
\begin{align*}
Q^{\varepsilon_k}_1= &\EE\left|\int^t_{t_0} B(s/\varepsilon_k,X^{\varepsilon_k}_s, Y^{\varepsilon_k}_s)ds\right|\\
\le &\EE\left|\int^t_{t_0}\Big[ B(s/\varepsilon_k,X^{\varepsilon_k}_s, Y^{\varepsilon_k}_s) -B(s/\varepsilon_k,X^{\varepsilon_k}_s, \hat Y^{\varepsilon_k}_s)\Big] ds\right|\\
&+\!\EE\left|\int^t_{t_0}\! \Big[B(s/\varepsilon_k,X^{\varepsilon_k}_s, \hat Y^{\varepsilon_k}_s) \!-\!B(s/\varepsilon_k,X^{\varepsilon_k}_{s(\delta)}, \hat Y^{\varepsilon_k}_s)\Big] ds\right|\!+\!\EE\left|\int^t_{t_0} \!\!B(s/\varepsilon_k,X^{\varepsilon_k}_{s(\delta)}, \hat Y^{\varepsilon_k}_s) ds\right|\\
=&:Q^{\varepsilon_k}_{11}+Q^{\varepsilon_k}_{12}+Q^{\varepsilon_k}_{13}.
\end{align*}
By \eqref{LipB} and the same argument used in the proof of Lemma \ref{DEY}, we have
\begin{align*}
Q^{\varepsilon_k}_{11}\le  &\left[\EE\left|\int^t_{t_0} B(s/\varepsilon_k,X^{\varepsilon_k}_s, Y^{\varepsilon_k}_s) -B(s/\varepsilon_k,X^{\varepsilon_k}_s, \hat Y^{\varepsilon_k}_s) ds\right|^2\right]^{1/2}\nonumber\\
\le  &C\Big[\EE\!\!\int_{t_0}^{t}\!\!\int_{r}^{t}\!\!\Big\langle B(s/\varepsilon_k,X_{s}^{\varepsilon_k},Y_{s}^{\varepsilon_k})\!-\!B(s/\varepsilon_k,X_{s}^{\varepsilon_k},\hat{Y}_{s}^{\varepsilon_k}), \!B(r/\varepsilon_k, X_{r}^{\varepsilon_k},Y_{r}^{\varepsilon_k})\!-\!B(r/\varepsilon_k,X_{r}^{\varepsilon_k},\hat{Y}_{r}^{\varepsilon_k})\Big\rangle ds dr\Big]^{1/2}\nonumber\\
\le  &C_T (1+|x|^{3/2}+|y|^{3/2}) \delta^{1/4}.
\end{align*}
According to \eqref{LipB}, \eref{IncrE}, \eqref{X} and Lemma \ref{MDY}, we have
\begin{align*}
Q^{\varepsilon_k}_{12}\le \int^t_{t_0} \EE\left[(1+|X^{\varepsilon_k}_s|+|\hat Y^{\varepsilon_k}_s|)|X^{\varepsilon_k}_s-X^{\varepsilon_k}_{s(\delta)}|\right]ds\le C_{T} (1+|x|^2+|y|^2)   \delta^{1/2}.
\end{align*}
Following the argument for \eqref{I2} in the proof of Lemma \ref{lemma 3.2}, we have
\begin{align*}
Q^{\varepsilon_k}_{13}\le \left[\EE\left|\int^t_{t_0} B(s/\varepsilon_k,X^{\varepsilon_k}_{s(\delta)}, \hat Y^{\varepsilon_k}_s) ds\right|^2\right]^{1/2}\le C_{T}(1+|x|+|y|)\left[\left(\varepsilon_k/\delta\right)^{1/2}+\delta\right].
\end{align*}
Therefore, putting all the estimates together, we obtain that
\begin{align}
Q^{\varepsilon_k}_1\le  C_{T}(1+|x|^2+|y|^2)\left[\left(\varepsilon_k/\delta\right)^{1/2} +\delta^{1/4}\right].\label{Q1}
\end{align}

For $Q^{\varepsilon_k}_2$, by \eref{F4.19} and the arguments for \eqref{E:add1} in the proof of Theorem \ref{main result 2}, we can easily obtain
\begin{align}
Q^{\varepsilon_k}_2 \le  C_T(1+|x|+|y|)\left[\phi_1(\delta/\varepsilon_k)+\phi_5(\delta/\varepsilon_k)+\delta\right].\label{Q2}
\end{align}

Hence, combining \eref{Q1} with \eref{Q2}, and taking $\delta=\varepsilon_k^{1/2}$, we obtain that \eref{F3.34} holds.

Thirdly, according to Assumption \ref{A5} and Lemma \ref{Lemma6.2},  for all $t\ge0$, $T>0$ and $x\in \RR^n$,
\begin{align*}
&\frac{1}{T} \left\|\int^{t+T}_t\left[\overline{\sigma\sigma^{\ast}}(s,x)-\bar{\Sigma}(x) \right]ds\right\|\nonumber\\
 & =\frac{1}{T}\left\|\int^{t+T}_t \left[\int_{\RR^m}(\sigma\sigma^{\ast})(s,x,y)\mu^x_s(dy)-\int_{\RR^m}\Sigma(x,y)\mu^x(dy)\right]ds\right\|\nonumber\\
  &\le \frac{1}{T}\left\|\int^{t+T}_t \left[\int_{\RR^m}(\sigma\sigma^{\ast})(s,x,y)\mu^x_s(dy)-\int_{\RR^m} (\sigma\sigma^{\ast})(s,x,y)\mu^x(dy)\right]ds\right\|\nonumber\\
&\quad + \frac{1}{T}\left\|\int^{t+T}_t\int_{\RR^m}\left[(\sigma\sigma^{\ast})(s,x,y)-\Sigma(x,y)\right]\mu^x(dy)ds\right\|ds\nonumber\\
 &=\frac{1}{T}\left\|\int^{t+T}_t \left[\int_{\RR^m}(\sigma\sigma^{\ast})(s,x,y)\mu^x_s(dy)-\int_{\RR^m} (\sigma\sigma^{\ast})(s,x,y)\mu^x(dy)\right]ds\right\|\nonumber\\
&\quad + \int_{\RR^m}\frac{1}{T}\left\|\int^{t+T}_t\left[(\sigma\sigma^{\ast})(s,x,y)-\Sigma(x,y)\right]ds\right\|\mu^x(dy)\nonumber\\
 & \le C(1+|x|)\frac{1}{T}\int^{t+T}_t\!\! \left(e^{-\beta s}+\int^s_0 e^{-\beta(s-u)}\phi_2(u)\,du\right) ds\!+\!C\phi_4(T)\int_{\RR^m} (1+|x|^2)\mu^x(dy)\nonumber\\
 &\le C(1+|x|^2)[\phi_4(T)+\tilde{\phi}_2(T)],
\end{align*}
where $\tilde{\phi}_2(T)$ is defined by \eref{phi5}. With the aid of this, one can follow the proof of  \eref{F3.34} to obtain that \eref{F3.35} also holds.

According to Lemma \ref{FianlW}, the averaged SDE \eref{e:ASDE2} has a unique strong solution $(X_t)_{t\ge0}$. Hence, the weak solution is also unique.
Since the uniqueness of the solution to the martingale problem with the operator $\bar{\mathcal{L}}$ is equivalent to the uniqueness of the weak solution to the SDE \eref{e:ASDE2}, the limiting process $(X_t)_{t\ge0}$ satisfying \eref{F3.33} is the unique solution to the SDE \eref{e:ASDE2} in the weak sense. The proof is complete. \end{proof}

\medskip

\noindent \textbf{Acknowledgment}. The research of Xiaobin Sun is supported by the NSF of China (Nos.
12271219, 12090010 and 12090011).
The research of Jian Wang is supported by the National Key R\&D Program of China (2022YFA1006003) and
the NSF of China (No.\ 12225104). The research of Yingchao Xie is supported by the NSF of China  (No.\ 12471139) and the Priority Academic Program Development of Jiangsu Higher Education Institutions.


\begin{thebibliography}{2}

\bibitem{B1999} P. Billingsley: \emph{Convergence of Probability Measuraes}, second edition,  John Wiley and Sons Inc., New York, 1999.

\bibitem{B2022} C.E. Br\'{e}hier: The averaging principle for stochastic differential equations driven by a Wiener process revisited, {\it C. R. Math. Acad. Sci. Paris}, 360 (2022), 265--273.


\bibitem{CWW2024} W. Cao, F. Wu and M. Wu: Weak convergence and stability of functional diffusion systems with singularly perturbed regime switching, {\it Nonlinear Anal. Hybrid Syst.}, 53 (2024), Paper no. 101487, 17 pp.

\bibitem{CL2017} S. Cerrai and A. Lunardi: Averaging principle for nonautonomous slow-fast systems of stochastic reaction-diffusion equations:
the almost periodic case, {\it SIAM J. Math. Anal.}, 49 (2017) 2843--2884.

\bibitem{Chen1992} M.-F. Chen: \emph{From Markov Chains to Non-Equilibrium Particle Systems}, second edition, World Scientific, Singapore, 2004.

\bibitem{CHR2024} M. Cheng, Z. Hao and M. R\"{o}ckner: Strong and weak convergence for the averaging principle of DDSDE with singular drift, {\it Bernoulli}, 30 (2024), 1586--1610.



\bibitem{DR2006} G. Da Prato and M. R\"{o}ckner: Dissipative stochastic equations in Hilbert space with time dependent coefficients, {\it Atti Accad. Naz. Lincei Rend. Lincei Mat. Appl.}, 17 (2006), 397--403.

\bibitem{DR2008} G. Da Prato and M. R\"{o}ckner: A note on evolution systems of measures for time-dependent stochastic differential equations, in: \emph{Seminar on Stochastic Analysis, Random Fields and Applications}, vol. 59, 2008, 115--122.

\bibitem{DEGZ} A. Durmus,  A. Eberle,  A. Guillin and  R. Zimmer:  An elementary approach to uniform in time propagation of chaos, {\it Proc. Amer. Math. Soc.}, {148} (2020),  5387--5398.

\bibitem{EE} W. E. and B. Engquist: Multiscale modeling and computations, {\it Notice of AMS}, 50 (2003), 1062--1070.

\bibitem{Eberle2014} A. Eberle: Reflection couplings and contraction rates for diffusions, {\it Probab. Theory Related Fields}, 166 (2016),  851--886.

\bibitem{FW2012} M. Freidlin and A.D. Wentzell: {\it Random Perturbations of Dynamical Systems,} third edition, Springer, Heidelberg, 2012.

\bibitem{GW2012} M. Galtier and G. Wainrib: Multiscale analysis of slow-fast neuronal learning models with noise, {\it J. Math. Neurosci.}, 2 (2012), 1--64.

\bibitem{GKK2006} D. Givon, I. G. Kevrekidis and R. Kupferman: Strong convergence of projective integeration schemes for singularly perturbed
stochastic differential systems, {\it Comm. Math. Sci.}, 4 (2006), 707--729.

\bibitem{HL2020} M. Hairer and X.-M. Li: Averaging dynamics driven by fractional Brownian motion, {\it Ann. Probab.}, 48 (2020), 1826--1860.

\bibitem{HLLS2023} W. Hong, S. Li, W. Liu and X. Sun: Central limit type theorem and large deviations for multi-scale McKean-Vlasov SDEs,  {\it Probab. Theory Related Fields} 187 (2023),  133--201.




\bibitem{K1968} R.Z. Khasminskii: On an averging principle for It\^{o} stochastic differential equations, {\it Kibernetica}, 4 (1968), 260--279.

\bibitem{L2010} D. Liu: Strong convergence of principle of averaging for multiscale stochastic dynamical systems, {\it Commun. Math. Sci.}, 8 (2010), 999--1020.


\bibitem{LR2015} W. Liu and M. R\"{o}ckner: \emph{Stochastic Partial Differential Equations: An Introduction}, Universitext, Springer, 2015.

\bibitem{LRSX2020} W. Liu, M. R\"{o}ckner, X. Sun and Y. Xie: Averaging principle for slow-fast stochastic differential equations with time dependent locally Lipschitz coefficients, {\it J. Differential Equations}, 268 (2020), 2910--2948.

\bibitem{LLZ2010} L. Lorenzi, A. Lunardi, and A. Zamboni: Asymptotic behavior in time periodic parabolic problems with unbounded coefficients, {\it Journal of Differential Equations}, 249 (2010), 3377--3418.

    \bibitem{LW2016} D. Luo and J. Wang: Exponential convergence in $L^p$-Wasserstein distance for diffusion processes without uniformly dissipative drift, {\it Math. Nachr.}, 289 (2016), 1909--1926.

\bibitem{PS2008} G. A. Pavliotis and A.M. Stuart: \emph{Multiscale Methods: Averaging and Homogenization}, Springer, New York, 2008.

\bibitem{PIX2021}  B. Pei, Y. Inahama and Y. Xu: Averaging principle for fast-slow system driven by mixed fractional Brownian rough path,  {\it J. Differential Equations}, 301 (2021), 202--235.

\bibitem{PW2006} E. Priola and F.-Y. Wang: Gradient estimates for diffusion semigroups with singular coefficient, {\it J. Funct. Anal.}, {236} (2006), 244--264.



\bibitem{RSX2021} M. R\"{o}ckner, X. Sun and Y. Xie: Strong convergence order for slow-fast McKean-Vlasov stochastic differential equations,  {\it Ann. Inst. Henri Poincar\'e Probab. Stat.}, 57 (2021), 547--576.

\bibitem{SXW2022} G. Shen, J. Xiang and J.-L. Wu: Averaging principle for distribution dependent stochastic differential equations driven by fractional Brownian motion and standard Brownian motion, {\it J. Differential Equations}, 321 (2022), 381--414.

\bibitem {SSWX2024} Y. Shi, X. Sun, L. Wang and Y. Xie:  Asymptotic behavior for multi-scale SDEs with monotonicity coefficients driven by L\'evy processes, {\it Potential Anal.}, 61 (2024), 111--152.

\bibitem {SWX2024} X. Sun, J. Wang and Y. Xie: Averaging principles for time-inhomogeneous multi-scale SDEs via nonautonomous Poisson equations, arXiv:2412.09850

\bibitem {SXX2022}  X. Sun, L. Xie and Y. Xie: Strong and weak convergence rates for slow-fast stochastic differential equations driven by $\alpha$-stable process, {\it Bernoulli}, 28 (2022), 343--369.

\bibitem{U2021} K. Uda: Averaging principle for stochastic differential equations in the random periodic regime, {\it Stochastic Process. Appl.}, 139 (2021), 1--36.



\bibitem{V1991} A.Y. Veretennikov: On the averaging principle for systems of stochastic differential equations, {\it Math. USSR Sborn}, 69 (1991), 271--284.

\bibitem{W2013} G. Wainrib: Double averaging principle for periodically forced slow-fast stochastic systems, {\it Electron. Commun. Probab.}, 18 (2013), 1--12.


\bibitem{Wang2} F.-Y. Wang:  Exponential ergodicity for non-dissipative McKean-Vlasov SDEs, {\it  Bernoulli}, {29} (2023),  1035--1062.

\bibitem{XLM2018} J. Xu, J. Liu and Y. Miao: Strong averaging principle for two-time-scale SDEs with non-Lipschitz coefficients, {\it J. Math. Anal. Appl.}, 468 (2018), 116--140.
\end{thebibliography}
\end{document}